\documentclass[preprint,10pt]{elsarticle}

\usepackage{amssymb}
\usepackage{graphicx}
\usepackage{pdflscape}
\usepackage{longtable}
\usepackage{lscape}
\usepackage{afterpage}
\usepackage{float} 
\usepackage{amsmath}
\allowdisplaybreaks
\usepackage[colorlinks=true]{hyperref}
\hypersetup{urlcolor=blue, citecolor=red}

\journal{Computational and Applied Mathematics }
\newtheorem{theorem}{Theorem}[section]
\newtheorem{corollary}[theorem]{Corollary}
\newtheorem{lemma}[theorem]{Lemma}
\newtheorem{proposition}[theorem]{Proposition}

\newtheorem{definition}[theorem]{Definition}
\newtheorem{remark}[theorem]{Remark}
\newtheorem{proof}{Proof}

\newcommand{\poubelle}[1]{}
\begin{document}

\begin{frontmatter}

\title{Exploring Well-Posedness and Asymptotic Behavior in an Advection-Diffusion-Reaction (ADR) Model}

\author[a,b,d]{Mohammed Elghandouri\corref{cor1}}
\ead{medelghandouri@gmail.com}
\author[b,d]{Khalil Ezzinbi}
\ead{ezzinbi@uca.ac.ma}
\author[b,c]{Lamiae Saidi}
\ead{maymia1987@gmail.com}
\cortext[cor1]{Corresponding author}
\affiliation[a]{organization={Centre INRIA de Lyon}, addressline={CEI-2 56, Boulevard Niels Bohr}, postcode={69 603}, city={Villeurbanne}, country={France}}
\affiliation[b]{organization={UCA-FSSM}, addressline={Bd. Prince My Abdellah, B.P. 2390}, postcode={40 000}, city={Marrakesh}, country={Morocco}}
\affiliation[c]{organization={UM6P-CRSA}, addressline={Lot 660, Hay Moulay Rachid}, postcode={43 150}, city={Ben Guerir}, country={Morocco}}
\affiliation[d]{organization={IRD-UMMISCO}, addressline={32 Av Henri Varagnat}, postcode={93 143}, city={Bondy}, country={France}}

\begin{abstract}
In this paper, the existence, uniqueness, and positivity of solutions, as well as the asymptotic behavior through a finite fractal dimensional global attractor for a general Advection-Diffusion-Reaction (ADR) equation, are investigated. Our findings are innovative, as we employ semigroups and global attractors theories to achieve these results. Also, an analytical solution of a two-dimensional Advection-Diffusion Equation is presented. And finally, two Explicit Finite Difference schemes are used to simulate solutions in the two- and three-dimensional cases. The numerical simulations are conducted with predefined initial and Dirichlet boundary conditions. 
\end{abstract}


\end{frontmatter}

\section{Introduction}
\noindent

The study of Partial Differential Equations (PDEs) holds a paramount position in the realm of mathematical analysis, finding applications across a diverse range of scientific disciplines and engineering fields. Among these, the Advection–Diffusion–Reaction (ADR) partial differential equations, that have widespread applications in fluid dynamics, heat or mass transfer, chemical reaction processes, contaminant transport and population dynamics \cite{Bear, Genuchten, Hetrick, Hundsdorfer, Murray, Shih, Steinfeld}. They model the temporal evolution of biological species in a flowing medium such as water or air, or contaminant transport with biological, chemical or radioactive processes, accounting for advection, diffusion and reaction processes. Therefore, there exists a wide literature about its resolution, numerically and analytically, surrounding the existence, uniqueness, asymptotic behavior, controllability,
numerical approximation of its solutions, etc. 
In  \cite{Zavaleta}, a distinctive classical solution for the single-step irreversible exothermic Arrhenius-type reaction within an incompressible fluid subject to Neumann boundary conditions is demonstrated. The authors provide both proof of the solution's existence and its uniqueness through the utilization of a semigroup formulation as well as the application of the maximum principle. In \cite{Peng}, the authors study the existence and uniqueness of mild solutions for a fractional reaction-diffusion equation, namely fractional Fokker–Planck equations. They use semigroup theory to establish their results. In \cite{MCLEAN}, the authors establish the well-posedness of an initial-boundary value problem encompassing a broad range of linear time-fractional advection-diffusion-reaction equations. The analysis hinges on innovative energy techniques, coupled with the application of a fractional Gronwall inequality and characteristics of fractional integrals. In \cite{ELAMVAZHUTHI}, the authors explore the precise characteristics of controllability exhibited by an advection–diffusion equation within a bounded domain. They achieve this investigation by utilizing time- and space-dependent velocity fields as controlling factors. Other relevant works are mentioned in \cite{LANSER, Polyanin, Schmidt, ZHANG, WHEELER}. Nevertheless, it's important to note that, in terms of mathematical advancements such as asymptotic behavior, existence, uniqueness, and positiveness of solutions for general ADR equations, there hasn't been significant development, particularly concerning semigroup methods.  

Numerically, there are various methods to solve different partial differential equations considering initial and boundary conditions like Dirichlet, Neumann, or Robin. A brief list of some numerical methods is: explicit methods (Finite Difference Methods \cite{Burden,Gustafson,Mphephu} (Explicit Forward Euler method, Upwind scheme, the Central Finite Difference scheme, Total Variation Diminishing (TVD) Schemes \cite{Shu}, etc.) and method of lines); Implicit methods (implicit backward Euler method, Crank Nicholson, Method of Lines); Iterative methods (Jacobi, Gauss–Seidel, and relaxation methods); and Finite Element Methods (Galerkin Method) \cite{Lewis}. However, the numerical resolution of a reaction-advection-diffusion equation while maintaining certain global physical properties such as mass conservation, positivity, and causality is complex. The goal is for the numerical solutions to satisfy the same properties as the exact solution such as positivity, boundedness, or monotonicity. Most numerical schemes aren't inherently designed to satisfy these properties, and they can lead to various numerical artifacts such as artificial oscillations and numerical dispersion. Explicit schemes are efficient but suffer from stability issues and require small time steps to meet the Courant-Friedrichs-Lewy (CFL) condition, thus the overall computational cost is usually high. Implicit schemes are unconditionally stable and allow larger time steps but involve solving algebraic systems at each time step. Semi-implicit schemes strike a balance between explicit and implicit methods. The choice of which scheme to use depends on the specific problem and requirements. Implicit schemes are often preferred for stiff initial value problems due to their stability and larger allowable time steps. Nevertheless, among the various numerical methods for solving ADR equations, finite difference method (FDM) seems to be more popular for the ease of implementation and its simplicity.

Our primary contribution in this work is the comprehensive investigation of the global existence, uniqueness, and positivity of solutions, as well as the exploration of the asymptotic behavior of an ADR equation through the application of semigroup theory. As a result, we establish the existence of a global attractor and determine its finite fractal dimension. To implement our theoretical findings numerically, we employ two fully explicit finite difference numerical schemes for solving the ADR equation: a 2-D explicit centered scheme and a 3-D explicit upwind scheme.

The organization of this paper is as follows: In Section~\ref{sec 1}, we present the model formulation, its origins, and the associated problems. In Section~\ref{sec 2}, we review useful concepts that will assist in proving our main results. Section~\ref{sec 3} employs semigroup theory to demonstrate the main findings related to the global existence, uniqueness, and positivity of solutions for our model. The existence of a global attractor with a finite fractal dimension is established in Section~\ref{section 6}. Section~\ref{sec 4} is dedicated to the analytical solution of the advection-diffusion equation in 2-D, while Section~\ref{sec 5} focuses on numerical simulations using finite difference methods.

\section{Formulation of the model} \label{sec 1}
\subsection{Origins and Problematic}
\noindent

The ADR model, which describes the time evolution of chemical species in air, is characterized by Partial Differential Equations (PDEs) derived from mass balances. Consider a concentration $c(t,x)$ of a certain chemical species, with space variable $x\in \Omega\subset\mathbb{R}$ and time $t\geq 0$. Let $h>0$ be a small number, and consider the average concentration $\bar{c}(t,x)$ in a cell $\Omega(x)=\left[x-\frac{1}{2} h, x+\frac{1}{2} h\right]$,
$$
\bar{c}(t,x)=\frac{1}{h} \int_{x-h / 2}^{x+h / 2} c\left(t,y\right) d y=c(t,x)+\frac{h^2}{24} \frac{\partial^2}{\partial x^2}c(t,x)+\cdots .
$$
If the species is carried along by a flowing medium with velocity $u(t,x)$ then the mass conservation law implies that the change of $\bar{c}(t,x)$ per unit of time is the net balance of inflow and outflow over the cell boundaries,
$$
\frac{\partial}{\partial t} \bar{c}(t,x)=\frac{1}{h}\left[u\left(t,x-\frac{1}{2} h\right) c\left(t,x-\frac{1}{2} h\right)-u\left(t,x+\frac{1}{2} h\right) c\left(t,x+\frac{1}{2} h\right)\right],
$$
where $u\left(t,x \pm \frac{1}{2} h\right) c\left(t,x \pm \frac{1}{2} h\right)$ are the mass fluxes over the left and right cell boundaries. Now, if we let $h \rightarrow 0$, it follows that the concentration satisfies
$$
\frac{\partial}{\partial t} c(t,x)+\frac{\partial}{\partial x}(u(t,x) c(t,x))=0.
$$
This is called an advection equation (or convection equation). In a similar way, we can consider the effect of diffusion. The change in $\bar{c}(t,x)$ is induced by gradients in the solution, and the fluxes across the cell boundaries are $-k\left(t,x \pm \frac{1}{2} h\right) \frac{\partial}{\partial x} c\left(t,x \pm \frac{1}{2} h\right)$, where $k(t,x)$ represents the diffusion coefficient. The corresponding diffusion equation is
$$
\frac{\partial}{\partial t} c(t,x)=\frac{\partial}{\partial x}\left(k(t,x) \frac{\partial}{\partial x} c(t,x)\right).
$$
There may also be a change in $c(t,x)$ due to sources, sinks and chemical reactions, leading to
$$
\frac{\partial}{\partial x}(u(t,x) c(t,x))=f(c(t,x),t,x).
$$
The overall change in concentration is described by combining these three effects, leading to the following Advection-Diffusion-Reaction (ADR) equation:
$$
\frac{\partial}{\partial t} c(t,x)+\frac{\partial}{\partial x}(u(t,x) c(t,x))=\frac{\partial}{\partial x}\left(k(t,x) \frac{\partial}{\partial x} c(t,x)\right)+f(c(t,x),t,x) .
$$
We shall consider the equation in a spatial domain $\Omega \subset \mathbb{R}^m$ ($m=1,2 \text{ or } 3$) with time $t \geq 0$. An initial profile $c(0,x)$ will be given and we also assume that suitable boundary conditions are provided.
More generally, let \( c_1(t,x), \dots, c_s(t,x) \) denote the concentrations of \( s \) chemical species, where the spatial variable \( x \in \Omega \) and time \( t \geq 0 \). Then the basic mathematical equations for transport and reaction are given by the following set of (PDEs):
$$
\begin{gathered}
\frac{\partial}{\partial t} c_j(t,x)+\sum_{i=1}^m \frac{\partial}{\partial x_i}\left(u_i(t,x) c_j(t,x)\right)= \\
\sum_{i=1}^m \frac{\partial}{\partial x_i}\left(k_i(t,x) \frac{\partial}{\partial x_i} c_j(t,x)\right)\!+\!f_j\left(c_1(t,x), \ldots, c_s(t,x),t,x\right),\quad j=1,2, \ldots, s,
\end{gathered}
$$
with suitable initial and boundary conditions. The quantities $u_i$ that represent the velocities of the transport medium, such as water or air, are either given in a data archive or computed alongside with a meteorological or hydrodynamical code. (In such codes Navier-Stokes or shallow water equations are solved, where again advection-diffusion equations are of primary importance.) The diffusion coefficients $k_i$ are constructed by the modelers and may include also parameterizations of turbulence. The final term $f_j(c, x, t)$, which gives a coupling between the various species, describes the nonlinear chemistry together with emissions (sources) and depositions (sinks). In actual models these equations are augmented with other suitable subgrid parameterizations and coordinate transformations.

\subsection{Description of the model}
\noindent

In this work, we undertake a comprehensive exploration of the following ADR equation:
\begin{equation}
\dfrac{\partial c_j(t,x)}{\partial t}=\sum\limits_{i=1}^{m}\left(k_i\dfrac{\partial^2 c_j(t,x)}{\partial x_i^2}-u_i\dfrac{\partial c_j(t,x)}{\partial x_i} \right)+R_j(t,c(t,x)),
\label{Eq 0}
\end{equation}
for $t\geq 0$, $x=(x_1,\ldots,x_m)\in\Omega$, where $\Omega$ is a bounded set in $\mathbb{R}^m$, $m\in\{1,2,3\}$, and $j=1,\ldots,s$ with $s\in \mathbb{N}^{*}$. The coefficients $k_i$ and $u_i$ are positive constants, and $c(t,x)=(c_1(t,x),\ldots,c_s(t,x))$. The function $R_j(t,c(t,x))$ represents the $j^{th}$ chemical reaction between species and is given by

\begin{equation}
R_j(t,c(t,x))=\sum\limits_{\kappa=1}^{r}(r_{j\kappa}-l_{j\kappa})h_\kappa(t)\prod\limits_{\nu=1}^{s}(c_\nu(t,x))^{l_{\nu\kappa}},
\hspace{0.1cm} j=1,\ldots,s, \hspace{0.1cm} x\in \Omega, \hspace{0.1cm} t\geq 0.
\label{Eq 1-2}
\end{equation}
Here, $r\in \mathbb{N}^{*}$ represents the number of chemical reactions between species, $l_{j\kappa}$ and $r_{j\kappa}$ are non-negative integers (belong to $\mathbb{N}$) describing the loss and gain of the number of molecules in the $j^{th}$ reaction, and $h_\kappa(t)$ is the rate of reaction, dependent on time $t$ due to external influences such as temperature and sunlight. Importantly, for each $\kappa=1,\ldots,r$, $h_\kappa(\cdot)$ is almost everywhere continuous, and there exists $d_\kappa>0$ such that $0\leq h_\kappa(t)\leq d_\kappa$ for $t\geq 0$. On the boundary of $\Omega$, we use Dirichlet boundary conditions, i.e.,
\begin{equation}
c_j(t,x)=0, \hspace{0.1cm} j=1,\ldots,s \hspace{0.1cm}\text{ for } \hspace{0.1cm} x\in \partial \Omega \hspace{0.1cm} \text{ and } \hspace{0.1cm} t\geq 0.
\label{Eq 2}
\end{equation}
As for the initial conditions, we denote them by
\begin{equation}
c_j(0,x)=c_{j,0}(x), \hspace{0.1cm} j=1,\ldots,s \hspace{0.1cm} \text{ for } \hspace{0.1cm} x\in \Omega,
\label{Eq 3.3}
\end{equation}
which represent the initial concentrations of the species.

 Let $\mathcal{A}:\left(H_0^{1}(\Omega)\cap H^{2}(\Omega) \right)^{s} \to \left(\mathbb{L}^{2}(\Omega)\right)^{s}$ be the map defined by
\begin{equation*}
\left(\mathcal{A} z\right)(x)=\sum\limits_{i=1}^{m}\left( k_i\dfrac{\partial^2 z(x)}{\partial x_i^2}-u_i\dfrac{\partial z(x)}{\partial x_i}\right) \text{ for } z\in \left(H_0^{1}(\Omega)\cap H^{2}(\Omega) \right)^{s} \text{ and } x\in \Omega.
\end{equation*}
We denote by $u(t)(x)=(u_1(t)(x),\ldots,u_s(t)(x)):=(c_1(t,x),\ldots,c_s(t,x))=c(t,x)$, and let $F(t,u(t))(x)=\left(F_1(t,u(t))(x),\ldots, F_s(t,u(t))(x)\right)$ for $t\geq 0$, where
\begin{equation}
F_j(t,u(t))(x)=R_j(t,u(t)(x)),\quad j=1, \ldots,s, \quad x\in \Omega,
\label{F}
\end{equation}
for $u(t)=(u_1(t),\ldots,u_s(t))\in D(F)$, where
\begin{equation}
D(F):=\left\{z=(z_1,\ldots,z_s)\in \left(\mathbb{L}^{2}(\Omega)\right)^{s}\mid F(t,z)\in \left(\mathbb{L}^{2}(\Omega)\right)^{s} \text{ for all } t\geq 0\right\}.
\end{equation}
Under the above notations, equations \eqref{Eq 0}--\eqref{Eq 3.3} can be expressed in the following abstract form:
\begin{equation}
\left\{
\begin{array}{l}
u'(t)=\mathcal{A} u(t)+ F(t,u(t)), \quad t\geq 0\\
u(0)=u_0.
\end{array}\right.
\label{Eq 1}
\end{equation}

We then delve into the mathematical properties of equation \eqref{Eq 1}, investigating conditions that ensure the existence, uniqueness and positivity of solutions, as well as exploring the asymptotic behavior through a finite fractal-dimensional global attractor. Additionally, we develop numerical strategies for effective approximation. Numerical simulations are conducted using the explicit forward Euler method for the three-dimensional advection-diffusion-reaction equations with Dirichlet boundary conditions set to zero. Our focus is on an advection-dominant problem related to the transport of air pollutants, providing valuable insights into the practical implications of our theoretical framework.
\section{Basic working tools} \label{sec 2}
\noindent

In this section, some necessary definitions, theorems, and lemmas to demonstrate our main results are recalled. Let $(X, \Vert\cdot\Vert_X)$ be a real Banach lattice space endowed with the ordering $\geqq$. We define $X_{+}$ as the set of all elements $x$ in $X$ such that $x\geqq 0$; this set is called the positive cone of $X$. Additionally, we denote $x^{+} = \max(0, x)$ for $x\in X$.

\begin{definition} \cite[Definition 11.9]{Batkai}
	A linear operator $A$ on $X$ is called dispersive if for every $f\in D(A)$, and $\lambda>0$ 
	\begin{center}
		$\Vert(\lambda f-Af)^{+}\Vert_X\geq \lambda\Vert f^{+}\Vert_X$.
	\end{center}
\end{definition}

The following Theorem shows the dispersivity of linear operators in Hilbert spaces.

\begin{theorem} \cite{Batkai} A linear operator $A$ on a real Hilbert lattice space $H$ is dispersive if and only if $\langle Af,f^+\rangle_{H}\leq 0$ for every $f\in D(A)$.
	\label{Remark 11.13}
\end{theorem}

We state the following definition of $A$-boundedness.
\begin{definition} \cite{Batkai} Let $A$ be a closed linear operator on $X$. A linear operator $B$ with a domain $D(B)$ is called $A$-bounded if $D(A)\subseteq D(B)$ and there exist $a,b\geq 0$ such that for all $f\in D(A)$ the inequality 
	\begin{equation}
	\Vert Bf\Vert_X\leq a\Vert A f\Vert_X+b\Vert f\Vert_X,
	\label{In 6}
	\end{equation}
	holds. The $A$-bound of $B$ is defined by
	\begin{center}
		$a_0=\inf\left\{a\geq 0\mid \text{there exists $b \geq 0$ such that the inequality \eqref{In 6} holds} \right\}$.
	\end{center}
\end{definition}

The following Theorem will be needed to demonstrate that the linear operator $\mathcal{A}$ generates a positive $C_0$-semigroup of contraction on $\left(\mathbb{L}^{2}(\Omega)\right)^{s}$.

\begin{theorem} \cite[Theorem 13.3.]{Batkai}
	Let $A$ be the generator of a positive $C_0$-semigroup of contraction on $X$ and $B$ a dispersive and $A$-bounded operator with $A$-bound $a_0<1$. Then $A+B$, defined on $D(A)$, generates a positive $C_0$-semigroup of contraction on $X$.
	\label{Theorem 13.3}
\end{theorem}

We recall the following definition of mild solutions for equation \eqref{Eq 1}.

\begin{definition} Let $0<a\leq +\infty$.
	A function $u(\cdot):[0,a)\to \left(\mathbb{L}^{2}(\Omega)\right)^{s}$ is called a mild solution of equation \eqref{Eq 1} on $[0,a)$ if
	\begin{equation}
	u(t)=S(t)u_0+\displaystyle\int_{0}^{t}S(t-\tau)F(\tau,u(\tau))d\tau \hspace{0.1cm} \text{ for } \hspace{0.1cm} t\in[0,a),
	\label{Eq 3}
	\end{equation}
	where $(S(t))_{t\geq 0}$ is the $C_0$-semigroup generated by $\mathcal{A}$ on $\left(\mathbb{L}^{2}(\Omega)\right)^{s}$.
\end{definition}

To demonstrate the existence of solutions for equation \eqref{Eq 1}, we require the following theorem.

\begin{theorem} \cite[Theorem 1.4]{Pazy} Let $f:\mathbb{R}^{+}\times X\to X$ be continuous w.r.t the first argument and locally Lipschitz continuous w.r.t the second argument. If $B$ is the infinitesimal generator of a $C_0$-semigroup $(T(t))_{t\geq 0}$ on $X$, then for every $v_0\in X$, there exists a $t_{\text{max}}\leq +\infty$ such that  the initial value problem:
	\begin{equation}
	\left\{\begin{array}{l}
	v'(t)=B v(t)+f(t,v(t)), \quad t\geq 0\\
	v(0)=v_0,
	\end{array}\right.
	\label{Eq 8}
	\end{equation}
	has a unique mild solution $v$ on $[0,t_{\text{max}})$. Moreover, if $t_{\text{max}}<+\infty$, then $\limsup\limits_{t\to t_{\text{max}}}\Vert v(t)\Vert_{X}=+\infty$.
	\label{existence}
\end{theorem}

\begin{corollary} \cite{Batkai} Assume that conditions in Theorem \ref{existence} are satisfied.
	If $f$ is at most affine w.r.t the second argument, then $t_{\text{max}}=+\infty$.
	\label{affine}
\end{corollary}
\begin{remark}\label{remark 3.8}
	Based on the proofs of Theorems 1.4 and 1.2 in \cite{Pazy}, it follows that if $f$ is almost everywhere continuous and bounded with respect to the first argument, and locally Lipschitz continuous with respect to the second argument, then the conclusions of Theorem \ref{existence} still hold.
\end{remark}

The following Lemmas will be needed throughout this work.

\begin{lemma} \cite[Proposition A.47]{Batkai}
	Let $I\subset \mathbb{R}$ be an open interval and $f\in W^{2,p}_{0}(I;\mathbb{R})$, $p\geq 1$. Then, for any
	$\epsilon>0$ there is a constant $C(\epsilon)>0$ such that
	\begin{center}
		$	\Vert f^{'}\Vert_{\mathbb{L}^{p}}\leq \epsilon\Vert f^{''}\Vert_{\mathbb{L}^{p}}+C(\epsilon)\Vert f\Vert_{\mathbb{L}^{p}}$.
	\end{center}
	\label{Proposition A.47}
\end{lemma}

\begin{lemma}  \cite[Theorem 11.3 with $h=0$, page 99]{Bainov} \itshape Let $v$ be a real, continuous and nonnegative function such that
	\begin{center}
		$v(t)\leq c+\displaystyle\int_{t_0}^{t}w(t,\tau)v(\tau)d\tau$ \text{ for } $t\geq t_0$,
	\end{center}
	where $c>0$, $w(t,\tau)$ is continuously differentiable in $t$ and continuous in $\tau$ with $w(t,\tau)\geq 0$ for $t\geq \tau\geq t_0$.
	Then,
	\begin{center}
		$v(t)\leq c\exp\left(\displaystyle\int_{t_0}^{t}\left[w(\tau,\tau)+\displaystyle\int_{t_0}^{\tau}\dfrac{\partial w(\tau,r)}{\partial \tau} dr\right]d\tau\right)$ \text{ for } $t\geq t_0$.
	\end{center}
	\label{lem 2.2}
\end{lemma}

\section{Well-posedness} \label{sec 3}
\noindent

In this section, we study the existence, uniqueness, and positiveness of the solution for equation \eqref{Eq 1}. The Hilbert lattice space $\left(\mathbb{L}^{2}(\Omega)\right)^{s}$ is equipped with the norm $\Vert\cdot\Vert$ defined by
\begin{center}
	$\Vert z\Vert=\left(\sum\limits_{j=1}^{s}\Vert z_j\Vert^{2}_{\mathbb{L}^{2}(\Omega)}\right)^{\frac{1}{2}}$ \text{ for } $z\in \left(\mathbb{L}^{2}(\Omega)\right)^{s}$.
\end{center}
We denote by $\left(\mathbb{L}^{2}(\Omega)\right)^{s}_{+}=\left\{z\in \left(\mathbb{L}^{2}(\Omega)\right)^{s}\mid z_j\geq 0, \hspace{0.1cm} j=1,\ldots,s\right\}$ the positive cone of the space $\left(\mathbb{L}^{2}(\Omega)\right)^{s}$.
\begin{proposition}
	$\mathcal{A}$ is an infinitesimal generator of a positive $C_0$-semigroup $(S(t))_{t\geq 0}$ of contraction on $\left(\mathbb{L}^{2}(\Omega)\right)^{s}$. 
	\label{positive}
\end{proposition}
\begin{proof}
	The linear operator $\mathcal{A}$ can be decomposed as follows:
	\begin{center}
		$\mathcal{A}z=\mathcal{A}_1z+\mathcal{A}_2z$ \text{ for } $z\in \left(H_0^{1}(\Omega)\cap H^{2}(\Omega) \right)^{s}$, 
	\end{center}
	where 
	\begin{center}
		$(\mathcal{A}_1z)(x)=\sum\limits_{i=1}^{m}\left( k_i\dfrac{\partial^2 z(x)}{\partial x_i^2}\right)$, \quad $z\in \left(H_0^{1}(\Omega)\cap H^{2}(\Omega) \right)^{s}$, \quad $x\in \Omega$,
	\end{center}
	and 
	\begin{center}
		$(\mathcal{A}_2z)(x)=\sum\limits_{i=1}^{m}\left( -u_i\dfrac{\partial z(x)}{\partial x_i}\right)$, \quad $z\in \left(H_0^{1}(\Omega)\cap H^{2}(\Omega) \right)^{s}$, \quad $x\in \Omega$.
	\end{center}
	By Lemma \ref{Proposition A.47}, we show that $\mathcal{A}_2$ is an $\mathcal{A}_1$-bounded operator with an $\mathcal{A}_1$-bounded equal to $0$. Let $z=(z_1,\ldots,z_s)\in \left(H_0^{1}(\Omega)\cap H^{2}(\Omega)\right)^{s}$, $z^{+}=(z^{+}_1,\ldots,z^{+}_s)$, and $z^{-}=(z^{-}_1,\ldots,z^{-}_s)$ be such that $z_j^{+}(x)=\max(0,z_j(x))$ and $z_j^{-}(x)=\max(0,-z_j(x))$ for $j=1\ldots,s$, and $x\in \Omega$. Let $z(x)=z^{+}(x)-z^{-}(x)$, then
	\begin{eqnarray*}
		\langle \mathcal{A}_2z,z^{+}\rangle_{\left(\mathbb{L}^{2}(\Omega)\right)^{s}}&=& -\sum\limits_{j=1}^{s}\displaystyle\int_{\Omega}\left(\sum\limits_{i=1}^{m}u_i\dfrac{\partial z_j(x)}{\partial x_i} z_j^{+}(x)\right) dx\\
		&=& \dfrac{1}{2}\sum\limits_{j=1}^{s}\sum\limits_{i=1}^{m}u_i\displaystyle\int_{\Omega}\dfrac{\partial }{\partial x_i} \left(z_j^{+}(x)\right)^{2} dx\\
		&\leq & \dfrac{m}{2}\max(u_1,\ldots,u_m) \sum\limits_{j=1}^{s}\left(z_j^{+}(x)\right)^{2}_{\mid_{\partial \Omega}}\\
		&=& 0.
	\end{eqnarray*}
	By Theorem \ref{Remark 11.13}, we obtain that $\mathcal{A}_2$ is  dispersive. We know that $\mathcal{A}_1$ is an infinitesimal generator of a positive $C_0$-semigroup of contraction on $\left(\mathbb{L}^{2}(\Omega)\right)^{s}$. Thus,  by Theorem \ref{Theorem 13.3}, we obtain that $\mathcal{A}$ is an infinitesimal generator of a positive $C_0$-semigroup $(S(t))_{t\geq 0}$ of contraction on $\left(\mathbb{L}^{2}(\Omega)\right)^{s}$.
\end{proof}

From now and throughout the rest of this work, we consider the following assumption. This assumption means that the system involves only monomolecular reactions. In other words, the following must be verified: for each $\kappa \in \{1, \ldots, r\}$, $0 \leq \sum_{j=1}^{s}l_{j\kappa} \leq 1$ (see \cite{Hundsdorfer}). Since $l_{j\kappa} \in \mathbb{N}$, this is equivalent to:
\begin{enumerate}
	\item[$\textbf{(H)}$] For each $\kappa \in \{1, \ldots, r\}$, $l_{j\kappa} = 0$ for all $j \in \{1, \ldots, s\}$, or there exists a unique $j_0 \in \{1, \ldots, s\}$ such that $l_{j_{0}\kappa} = 1$ and $l_{j\kappa} = 0$ for all $j \in \{1, \ldots, s\} - \{j_0\}$.
\end{enumerate}

In such cases, the equation \eqref{Eq 1} becomes linear in $u$; moreover, $D(F) = (\mathbb{L}^{2}(\Omega))^{s}$. The following Theorem shows the global existence for equation \eqref{Eq 1}.

\begin{theorem} Assume that $\textbf{(H)}$ holds. Then, equation \eqref{Eq 1} has a unique global solution defined on $[0,+\infty[$.
	\label{thm 3.6}
\end{theorem}

The following Lemma will help us to demonstrate the result stated in Theorem \ref{thm 3.6}.
\begin{lemma} Assume that $\textbf{(H)}$ holds, then
	there exists $\bar{d}>0$ and $\beta\in\{0,1\}$ such that 
	$\Vert F(t,z)\Vert\leq \bar{d}\Vert z\Vert^{\beta}$ \text{ for each } $z\in (\mathbb{L}^{2}(\Omega))^{s}$.
	\label{Lemma 5.4}
\end{lemma}
\begin{proof}
	Let $z = (z_1, \ldots, z_s) \in \left(\mathbb{L}^{2}(\Omega)\right)^{s}$. Let $x \in \Omega$ and $t\geq 0$. We discuss two cases:
	\item i) case 1: if for each $\kappa\in\{1,\ldots,r\}$, $l_{j\kappa}=0$ for all $j\in\{1,\ldots,s\}$, then for each $j=1,\ldots,s$, we have 
	\begin{eqnarray*}
		\vert F_j(t,z)(x)\vert^{2}&\leq&\left(\sum\limits_{\kappa=1}^{r}r_{j\kappa}h_\kappa(t)\right)^{2}\leq 2(r-1)  \sum\limits_{\kappa=1}^{r}\left(r_{j\kappa}d_\kappa\right)^{2}.
	\end{eqnarray*}
	This implies that 
	\begin{eqnarray*}
		\Vert F_j(t,z)\Vert^{2}&\leq& 2(r-1)  \sum\limits_{\kappa=1}^{r}\left(r_{j\kappa}d_\kappa\right)^{2}.
	\end{eqnarray*}
	Thus, 
	\begin{eqnarray*}
		\Vert F(t,z)\Vert &\leq & \left(2(r-1)\sum\limits_{j=1}^{s}  \sum\limits_{\kappa=1}^{r}\left(r_{j\kappa}d_\kappa\right)^{2}\right)^{\frac{1}{2}} \Vert z\Vert^{\beta} \hspace{0.1cm}\text{ for } \hspace{0.1cm} \beta=0.
	\end{eqnarray*}
	\item ii) case 2: if for each $\kappa\in\{1,\ldots,r\}$, there exists a unique $j_0\in\{1,\ldots,s\}$ such that $l_{j_{0}\kappa}=1$ and $l_{j\kappa}=0$ for all $j\in\{1,\ldots,s\}-\{j_0\}$, then for each $j=1,\ldots,s$, we have
	
	\begin{eqnarray*}
		\vert F_j(t,z)(x)\vert^{2}&\leq&\left(\sum\limits_{\kappa=1}^{r}(r_{j\kappa}-l_{j\kappa})h_\kappa(t)\vert z_{j_{0}}(x)\vert\right)^{2}\\
		&\leq& 2(r-1)\sum\limits_{\kappa=1}^{r}\left[(r_{j\kappa}-l_{j\kappa})h_\kappa(t)\right]^{2}\vert z_{j_{0}}(x)\vert^{2}.
	\end{eqnarray*}
	This implies that 
	\begin{eqnarray*}
		\Vert F_j(t,z)\Vert^{2}
		&\leq& 2(r-1)\sum\limits_{\kappa=1}^{r}\left[(r_{j\kappa}-l_{j\kappa})h_\kappa(t)\right]^{2}\Vert z\Vert^{2}.
	\end{eqnarray*}
	Thus, 
	\begin{eqnarray*}
		\Vert F(t,z)\Vert\leq \left(2(r-1)\sum\limits_{j=1}^{s} \sum\limits_{\kappa=1}^{r}\left[\left(r_{j\kappa}-l_{j\kappa}\right)d_\kappa\right]^{2}\right)^{\frac{1}{2}}\Vert z\Vert^{\beta} \hspace{0.1cm} \text{ for } \hspace{0.1cm} \beta=1.
	\end{eqnarray*}
	In both cases, by denoting 
	\begin{center}
		$\bar{d}=\sqrt{2(r-1)}\max\left(\left(\sum\limits_{j=1}^{s}  \sum\limits_{\kappa=1}^{r}\left(r_{j\kappa}d_\kappa\right)^{2}\right)^{\frac{1}{2}},\left(\sum\limits_{j=1}^{s} \sum\limits_{\kappa=1}^{r}\left[(r_{j\kappa}-l_{j\kappa})d_\kappa\right]^{2}\right)^{\frac{1}{2}}\right)$,
	\end{center}
	we get that for each $z\in (\mathbb{L}^{2}(\Omega))^{s}$,
	\begin{eqnarray*}
		\Vert F(t,z)\Vert \leq \bar{d}\Vert z\Vert^{\beta} \hspace{0.1cm} \text{ for } \hspace{0.1cm} \beta\in\{0,1\}.
	\end{eqnarray*}
\end{proof}
\begin{proof}[Proof of Theorem \ref{thm 3.6}] Under the assumption $\mathbf{(H)}$, the function $F(t,z)$ is linear in $z \in (\mathbb{L}^{2}(\Omega))^{s}$ and almost everywhere continuous with respect to $t \geq 0$. Using Theorem \ref{existence} and Remark \ref{remark 3.8}, we show that equation \eqref{Eq 1} has a unique solution $u(\cdot)$ defined on $[0, t_{\text{max}})$ for some $t_{\text{max}} \leq +\infty$. To demonstrate that $t_{\text{max}} = +\infty$, we will employ Lemma \ref{Lemma 5.4} and discuss two cases:
	\item i) case 1: if $\beta= 0$, we obtain that 
	\begin{eqnarray*}
		\Vert u(t)\Vert &\leq& \Vert S(t)u_0\Vert+  \displaystyle\int_{0}^{t}\Vert S(t-\tau)F(\tau,u(\tau))\Vert d\tau \\
		& \leq & \Vert u_0\Vert+\displaystyle\int_{0}^{t}\Vert F(\tau,u(\tau))\Vert d\tau\\
		&\leq & \Vert u_0\Vert+ \bar{d} t
	\end{eqnarray*}
	If $t_{\text{max}}<+\infty$, we get that 
	\begin{center}
		$\limsup\limits_{t\to t_{\text{max}}}\Vert u(t)\Vert\leq \Vert u_0\Vert+ \bar{d} t_{\text{max}} <+\infty$.
	\end{center}
	In light of Theorem \ref{existence}, this leads to a contradiction, and therefore, we conclude that $t_{\text{max}}=+\infty$.
	\item ii) case 2: if $\beta=1$, then $F$ is at most affine w.r.t the second argument, using Corollary \ref{affine}, we obtain that $t_{\text{max}}=+\infty$.
\end{proof}

The following Theorem shows the positivity of solutions for equation \eqref{Eq 1} on $\left(\mathbb{L}^{2}(\Omega)\right)^{s}_{+}$.

\begin{theorem} Assume that $\textbf{(H)}$ holds, then the solutions of equation \eqref{Eq 1}, starting from $u_0\in \left(\mathbb{L}^{2}(\Omega)\right)^{s}_{+}$, remain in $\left(\mathbb{L}^{2}(\Omega)\right)^{s}_{+}$.
\end{theorem}

\begin{proof} We discuss two cases:
	\item i) case 1: if for each $\kappa\in\{1,\ldots,r\}$, $l_{j\kappa}=0$ for all $j\in\{1,\ldots,s\}$, then $F(t,z)\equiv F(t)$ doesn't depend in $z$ and $F(t)\geq 0$. By Theorem \ref{thm 3.6}, we get that equation \eqref{Eq 1} has a unique solution $u(\cdot)$ defined on $[0,+\infty[$ by
	\begin{center}
		$u(t)=S(t)u_0+\displaystyle\int_{0}^{t}S(t-\tau)F(\tau)d\tau$ \text{ for } $t\geq 0$.
	\end{center}
	Since $(S(t))_{t\geq 0}$ is positive, we obtain that $u(t)\in \left(\mathbb{L}^{2}(\Omega)\right)^{s}_{+}$.
	\item  ii) case 2: if for each $\kappa\in\{1,\ldots,r\}$, there exists a unique $j_0\in\{1,\ldots,s\}$ such that $l_{j_{0}\kappa}=1$ and $l_{j\kappa}=0$ for all $j\in\{1,\ldots,s\}-\{j_0\}$, then $F(t,z)$ is linear in $z$. Let $\lambda>0$ be sufficiently large such that $F(t,v)+\lambda v\geq 0$ for each $v\in \left(\mathbb{L}^{2}(\Omega)\right)^{s}_{+}$. Equation \eqref{Eq 1} can be written in the following equivalent form:
	\begin{equation}
	\left\{\begin{array}{l}
	u'(t)= (\mathcal{A} -\lambda I)u(t)+F(t,u(t))+\lambda u(t), \quad t\geq 0\\
	u(0)=u_0.
	\end{array}\right.
	\label{equation 5}
	\end{equation}
	The equation \eqref{equation 5} has a unique solution $u(\cdot)$ defined on $[0,+\infty[$ by
	\begin{center}
		$u(t)=e^{-\lambda t}S(t)u_0+\displaystyle\int_{0}^{t}e^{-\lambda(t-\tau)}S(t-\tau)\left[F(\tau,u(\tau))+\lambda u(\tau)\right]d\tau$ \text{ for } $t\geq 0$.
	\end{center}
	The result follows the fact that $(S(t))_{t\geq 0}$ is positive and $F(t,v)+\lambda v\geq 0$ for each $v\in \left(\mathbb{L}^{2}(\Omega)\right)^{s}_{+}$.
\end{proof}

\section{Global attractor} \label{section 6}
\noindent

Understanding the long-term behaviour of dynamical systems is a fundamental research challenge. A crucial concept in the study of the behaviour of such systems is that of global attractors \cite{Hale}. The fundamental results presented in the previous section form the cornerstone of our study of the long-term behaviour of the equation \eqref{Eq 1} using the theory of global attractors. In this section, we carry out a complete examination of the long-term dynamics of the equation \eqref{Eq 1}. We thus establish the existence of a global attractor for our model. 

Let us first revisit some key properties that are derived from the theory of global attractors and will help us in establishing our results in this section. For a more comprehensive understanding, readers are encouraged to refer to \cite{Hale}.

\begin{definition} \cite{Hale}
	A semiflow $(\mathcal{U}(t))_{t\geq 0}$ on a complete metric space $X$ is a one parameter family maps $\mathcal{U}(t):X\to X$, parameter $t\in \mathbf{R}^{+}$, such that  $\mathcal{U}(0)=id_{X}$ ($id_{X}$ is the identity map on $X$), $\mathcal{U}(t+s)=\mathcal{U}(t)\mathcal{U}(s)$ for $t,s\geq 0$, and $t\to \mathcal{U}(t)x$ is continuous for all $x\in X$.
\end{definition}

\begin{definition} \cite{Hale}
	Let $(\mathcal{U}(t))_{t\geq 0}$ be a semiflow on a complete metric space $(X,d)$.
	\begin{description}
		\item[i)] We say that $(\mathcal{U}(t))_{t\geq 0}$ is bounded, if it takes bounded sets into bounded sets.
		\item[ii)] We say that $B\subset X$ is invariant under $(\mathcal{U}(t))_{t\geq 0}$ if $\mathcal{U}(t)B= B$ for $t\geq 0$.
		\item[iii)] We say that $(\mathcal{U}(t))_{t\geq 0}$ is point dissipative if there is a bounded set $B\subset X$ which attracts each point of $X$ under $(\mathcal{U}(t))_{t\geq 0}$.
		\item[iv)] A subset $B$ of $X$ is said to attract  $C\subset X$ if $d(\mathcal{U}(t)C,B)\to 0$ as $t\to +\infty$, 
		where 
		\begin{center}
			$d(\mathcal{U}(t)C,B):=\inf\{ d(\mathcal{U}(t)x,y)\mid x\in B \text{ and } y\in C \}$.
		\end{center}
	\end{description}
\end{definition}

\begin{theorem} \cite{Hale} Let $(\mathcal{U}(t))_{t\geq 0}$ be a semiflow on a complete space $X$ with the metric $\Vert\cdot\Vert_X$. Assume that there exists $L>0$ such that for all $x_0\in X$, $\limsup\limits_{t\to +\infty}\Vert \mathcal{U}(t)x_0\Vert_{X}\leq L$, then $(\mathcal{U}(t))_{t\geq 0}$ is point dissipative.
	\label{Thm 4.11}
\end{theorem}

\begin{definition} \cite{Hale} Let $X$ be a complete metric space and $(\mathcal{U}(t))_{t\geq 0}$ be a semiflow on $X$. Let $\mathcal{M}$ be a subset of $X$. Then, $\mathcal{M}$ is called a global attractor for $(\mathcal{U}(t))_{t\geq 0}$ in $X$ if $\mathcal{M}$ is a closed and bounded invariant (under $(\mathcal{U}(t))_{t\geq 0}$) set of $X$ that attracts every bounded set of $X$.
\end{definition}

The following Theorem shows the existence of global attractors for point dissipative semiflows.

\begin{theorem} \cite[Theorem 4.1.2, page 63]{Hale} Let $(\mathcal{U}(t))_{t\geq 0}$ be a semiflow on a complete metric space $X$. Suppose that $\mathcal{U}(t)$ is bounded for each $t\geq 0$. If $(\mathcal{U}(t))_{t\geq 0}$ is point dissipative, then it has a connected global attractor.
	\label{Theorem 4.1.2}
\end{theorem}

\subsection{Why do we need to look for an attractor for equation \eqref{Eq 1}?}
\noindent

The set of (constant and non-constant) equilibrium points of equation \eqref{Eq 1} is given by:
\begin{center}
	$\mathcal{S}=\left\{u^{*}=(u_1^*,\ldots,u_s^{*})\in D(\mathcal{A})\mid \text{ such that } \mathcal{A}u^*+F(t,u^*)=0 \text{ for } t\geq 0 \right\}$.
\end{center}
As we can see, these equilibrium points are not isolated. This lack of isolation means that we cannot apply techniques such as linearization or Lyapunov's methods, typically used to study the asymptotic behavior of systems near isolated equilibrium points. Furthermore, we do not anticipate observing convergence towards these equilibrium points, which complicates the understanding of the asymptotic behavior of our system as it approaches these equilibrium points. Because of these challenges, our primary focus is on establishing the existence of a finite fractal dimensional global attractor for equation \eqref{Eq 1}.
\subsection{ Existence of a global attractor for equation \eqref{Eq 1} }
\noindent

Let us define the family $(\mathcal{U}(t))_{t\geq 0}$ on $(\mathbb{L}^{2}(\Omega))^{s}_{+}$ as follows: for $t\geq 0$ and $u_0\in (\mathbb{L}^{2}(\Omega))^{s}_{+}$, we set $\mathcal{U}(t)u_0:=u(t,u_0):=u(t)$, where $u(\cdot,u_0)$ represents the unique solution  of equation \eqref{Eq 1} corresponding to $u_0\in (\mathbb{L}^{2}(\Omega))^{s}_{+}$. It is clear that $\mathcal{U}(0)u_0=u_0$, and for each $u_0\in (\mathbb{L}^{2}(\Omega))^{s}_{+}$, the map $t\to \mathcal{U}(t)u_0$ is continuous. Additionally, due to the Lipschitz property of the function $F$ with respect to the second argument, we can demonstrate that $\mathcal{U}(t+\tau)u_0=\mathcal{U}(t)\mathcal{U}(\tau)u_0$ for all $t,\tau\geq 0$ and $u_0\in (\mathbb{L}^{2}(\Omega))^{s}_{+}$. Consequently, we conclude that the family $(\mathcal{U}(t))_{t\geq 0}$ constitutes a semiflow on $(\mathbb{L}^{2}(\Omega))^{s}_{+}$. The following Theorem shows the boundnesse of the semiflow $(\mathcal{U}(t))_{t\geq 0}$ on $(\mathbb{L}^{2}(\Omega))^{s}_{+}$.

\begin{theorem} Assume that $\textbf{(H)}$ holds, then $\mathcal{U}(t)$ is bounded for each $t\geq 0$. Moreover, 
	\begin{center}
		$\Vert \mathcal{U}(t)u_0\Vert\leq e^{\bar{d}t}\left[\Vert u_0\Vert+1\right]$ \text{ for } $t\geq 0$,  \text{ and } $u_0\in (\mathbb{L}^{2}(\Omega))^{s}_{+}$.
	\end{center}
	\label{bounded}
\end{theorem}
\begin{proof} We employ Lemma \ref{Lemma 5.4}.
	Under assumption $\textbf{(H)}$, we have two cases: $\beta=0$ or $\beta=1$. We recall that 
	\begin{center}
		$\mathcal{U}(t)u_0=S(t)u_0+\displaystyle\int_{0}^{t}S(t-\tau)F(\tau,\mathcal{U}(\tau)u_0)d\tau$ \text{ for } $t\geq 0$.
	\end{center}
	\begin{enumerate}
		\item[1)] Case 1: $\beta=0$. In this case, for $t\geq 0$ and $u_0\in (\mathbb{L}^{2}(\Omega))^{s}_{+}$, by Lemma \ref{Lemma 5.4}, we obtain that
		\begin{eqnarray*}
			\Vert \mathcal{U}(t)u_0\Vert &\leq& \Vert u_0\Vert+\displaystyle\int_{0}^{t}\Vert F(\tau,\mathcal{U}(\tau)u_0)\Vert d\tau\\
			&\leq & \Vert u_0\Vert+\bar{d}t.
		\end{eqnarray*}
		\item[2)] Case 2: $\beta=1$. In this case, for $t\geq 0$ and $u_0\in (\mathbb{L}^{2}(\Omega))^{s}_{+}$, by Lemma \ref{Lemma 5.4}, we obtain that
		\begin{eqnarray*}
			\Vert \mathcal{U}(t)u_0\Vert &\leq& \Vert u_0\Vert+\displaystyle\int_{0}^{t}\bar{d}\Vert \mathcal{U}(\tau)u_0\Vert d\tau.
		\end{eqnarray*} 
		By Gronwall's lemma, we conclude that 
		\begin{eqnarray*}
			\Vert \mathcal{U}(t)u_0\Vert &\leq& e^{\bar{d}t}\Vert u_0\Vert.
		\end{eqnarray*}
	\end{enumerate}
	In both cases, we have
	\begin{center}
		$\Vert \mathcal{U}(t)u_0\Vert\leq e^{\bar{d}t}\left[\Vert u_0\Vert+1\right]$ \text{ for } $t\geq 0$,  \text{ and } $u_0\in (\mathbb{L}^{2}(\Omega))^{s}_{+}$.
	\end{center}
\end{proof}

The following result shows the point dissipativeness of $(\mathcal{U}(t))_{t\geq 0}$ on $(\mathbb{L}^{2}(\Omega))^{s}_{+}$.
\begin{theorem} Assume that $\textbf{(H)}$ holds,
	the semiflow  $\mathcal{U}(t)_{t\geq 0}$ is point dissipative on $(\mathbb{L}^{2}(\Omega))^{s}_{+}$.
	\label{dissipative}
\end{theorem}
\begin{proof}
	As in the proof of Theorem \ref{bounded}, assumption $\textbf{(H)}$ imply that $\beta=0$ or $\beta=1$. Let $u_0\in (\mathbb{L}^{2}(\Omega))^{s}_{+}$ be fixed, $\lambda>1$ be sufficiently large, and $t_0:=t_0(u_0)>0$ such that
	\begin{equation}
	\dfrac{e^{(\bar{d}+\lambda)t_0} \left[\Vert u_0\Vert+1\right]}{e^{\lambda t}}\leq 1 \hspace{0.1cm} \text{ for } \hspace{0.1cm} t> t_0.
	\label{eq 10}
	\end{equation}
	We recall that equation \eqref{Eq 1} can take the equivalent form \eqref{equation 5}, and therefore:
	
	\begin{center}
		$\mathcal{U}(t)u_0=e^{-\lambda t}S(t)u_0\!+\!\displaystyle\int_{0}^{t}e^{-\lambda(t-\tau)}S(t-\tau)\left[F(\tau,\mathcal{U}(\tau)u_0)\!+\!\lambda \mathcal{U}(\tau)u_0\right]d\tau$,\hspace{0.2cm} $t\geq 0$.
	\end{center}
	\item i) If $\beta=0$, for $t\geq t_0$, using Lemma \ref{Lemma 5.4}, we have
	\begin{eqnarray*}
		\Vert \mathcal{U}(t)u_0\Vert & \leq & e^{-\lambda t}\Vert u_0\Vert+\displaystyle\int_{0}^{t}e^{-\lambda(t-\tau)}\left[\bar{d}+\lambda\Vert \mathcal{U}(\tau)u_0\Vert\right]d\tau\\
		&=& e^{-\lambda t}\Vert u_0\Vert+\bar{d}\displaystyle\int_{0}^{t}e^{-\lambda(t-\tau)}d\tau+\displaystyle\int_{0}^{t}\lambda e^{-\lambda(t-\tau)} \Vert \mathcal{U}(\tau)u_0\Vert d\tau\\
		&=& e^{-\lambda t}\Vert u_0\Vert+\frac{\bar{d}}{\lambda}\left[1-e^{-\lambda t}\right]+\displaystyle\int_{0}^{t}\lambda e^{-\lambda(t-\tau)} \Vert \mathcal{U}(\tau)u_0\Vert d\tau.
	\end{eqnarray*}
	By Theorem \ref{bounded}, we obtain that 
	\begin{eqnarray*}
		\Vert \mathcal{U}(t)u_0\Vert & \leq & J_{1,\lambda}(t_0,t)+\displaystyle\int_{t_0}^{t}\lambda e^{-\lambda(t-\tau)} \Vert \mathcal{U}(\tau)u_0\Vert d\tau,
	\end{eqnarray*}
	where
	\begin{eqnarray*}
		J_{1,\lambda}(t_0,t) &=& e^{-\lambda t}\Vert u_0\Vert+\frac{\bar{d}}{\lambda}\left[1-e^{-\lambda t}\right]+\lambda e^{-\lambda t}\displaystyle\int_{0}^{t_0} e^{(\bar{d}+\lambda)\tau}\left[\Vert u_0\Vert+1\right] d\tau\\
		&\leq & e^{-\lambda t}\Vert u_0\Vert+\frac{\bar{d}}{\lambda}\left[1-e^{-\lambda t}\right]+\frac{\lambda}{\bar{d}+\lambda}e^{-\lambda t}\left[\Vert u_0\Vert+1\right](e^{(\bar{d}+\lambda)t_0}-1)\\
		&\leq & \dfrac{e^{(\bar{d}+\lambda)t_0} \left[\Vert u_0\Vert+1\right]}{e^{\lambda t}}+ \frac{\bar{d}}{\lambda}.
	\end{eqnarray*}
	By \eqref{eq 10}, we get that 
	\begin{eqnarray*}
		\Vert \mathcal{U}(t)u_0\Vert & \leq & \left(1+\frac{\bar{d}}{\lambda}\right)+\displaystyle\int_{t_0}^{t}w(t,\tau) \Vert \mathcal{U}(\tau)u_0\Vert d\tau,
	\end{eqnarray*}
	where 
	\begin{center}
		$w(t,\tau)=\lambda e^{-\lambda(t-\tau)}$ \text{ for } $t\geq \tau>t_0$.
	\end{center}
	Using Lemma \ref{lem 2.2}, we obtain that
	\begin{eqnarray*}
		\Vert \mathcal{U}(t)u_0\Vert & \leq & \left(1+\frac{\bar{d}}{\lambda}\right)\exp\left(\displaystyle\int_{t_0}^{t}\left[\lambda-\lambda\displaystyle\int_{t_0}^{\tau}\lambda e^{-\lambda(\tau-r)}dr\right]d\tau\right)\\
		&=& \left(1+\frac{\bar{d}}{\lambda}\right)\exp\left(\displaystyle\int_{t_0}^{t}\left[\lambda-\lambda(1-e^{-\lambda(\tau-t_0)})\right]d\tau\right)\\
		&=& \left(1+\frac{\bar{d}}{\lambda}\right)\exp\left(\displaystyle\int_{t_0}^{t}\lambda e^{-\lambda(\tau-t_0)}d\tau\right)\\
		&=& \left(1+\frac{\bar{d}}{\lambda}\right)\exp\left(1-e^{-\lambda(t-t_0)}\right).
	\end{eqnarray*}
	\item ii) If $\beta=1$, for $t\geq t_0$, using Lemma \ref{Lemma 5.4}, we obtain that
	\begin{eqnarray*}
		\Vert \mathcal{U}(t)u_0\Vert & \leq & e^{-\lambda t}\Vert u_0\Vert+\displaystyle\int_{0}^{t}(\bar{d}+\lambda)e^{-\lambda(t-\tau)}\Vert \mathcal{U}(\tau)u_0\Vert d\tau.
	\end{eqnarray*}
	By Theorem \ref{bounded}, we get that 
	\begin{eqnarray*}
		\Vert \mathcal{U}(t)u_0\Vert & \leq & J_{2,\lambda}(t_0,t) + \displaystyle\int_{t_0}^{t}(\bar{d}+\lambda)e^{-\lambda(t-\tau)}\Vert \mathcal{U}(\tau)u_0\Vert d\tau,
	\end{eqnarray*}
	where 
	\begin{eqnarray*}
		J_{2,\lambda}(t_0,t) &=& e^{-\lambda t}\Vert u_0\Vert+e^{-\lambda t}\displaystyle\int_{0}^{t_0}(\bar{d}+\lambda)e^{(\lambda+\bar{d})\tau}\left[\Vert u_0\Vert+1\right] d\tau\\
		& = & e^{-\lambda t}\Vert u_0\Vert+e^{-\lambda t}\left[\Vert u_0\Vert+1\right]\left(e^{(\lambda+\bar{d})t_0}-1\right)\\
		&\leq & \dfrac{e^{(\bar{d}+\lambda)t_0} \left[\Vert u_0\Vert+1\right]}{e^{\lambda t}}.
	\end{eqnarray*}
	By \eqref{eq 10}, we have 
	\begin{eqnarray*}
		\Vert \mathcal{U}(t)u_0\Vert &\leq & 1 + \displaystyle\int_{0}^{t}w(t,\tau)\Vert \mathcal{U}(\tau)u_0\Vert d\tau,
	\end{eqnarray*}
	where 
	\begin{center}
		$ w(t,\tau)=(\bar{d}+\lambda)e^{-\lambda(t-\tau)}$ \text{ for } $t\geq \tau>t_0$.
	\end{center}
	Using Lemma \ref{lem 2.2}, we obtain that 
	\begin{eqnarray*}
		\Vert \mathcal{U}(t)u_0\Vert &\leq & \exp\left(\displaystyle\int_{t_0}^{t}\left[(\bar{d}+\lambda)-(\bar{d}+\lambda)\displaystyle\int_{t_0}^{\tau}\lambda e^{-\lambda(\tau-r)}dr\right]d\tau\right)\\
		&= & \exp\left(\displaystyle\int_{t_0}^{t}\left[(\bar{d}+\lambda)-(\bar{d}+\lambda)(1-e^{-\lambda (\tau-t_0)})\right]d\tau\right)\\
		&= & \exp\left(\dfrac{(\bar{d}+\lambda)}{\lambda}(1-e^{-\lambda (t-t_0)})\right).
	\end{eqnarray*}
	In both cases, we obtain that 
	\begin{center}
		$\Vert \mathcal{U}(t)u_0\Vert \leq (1+\bar{d})e^{(1+\bar{d})} $ \text{ for } $t>t_0$.
	\end{center}
	As a consequence, 
	\begin{center}
		$\limsup\limits_{t\to +\infty}\Vert \mathcal{U}(t)u_0\Vert \leq (1+\bar{d})e^{(1+\bar{d})}$.
	\end{center}
	By Theorem \ref{Thm 4.11}, we conclude that $(\mathcal{U}(t))_{t\geq 0}$ is point dissipative on $(\mathbb{L}^{2}(\Omega))^{s}_{+}$.
\end{proof}

The main result in this section is the following, which directly follows from Theorem \ref{Theorem 4.1.2}, Theorem \ref{bounded}, and Theorem \ref{dissipative}.
\begin{theorem} Assume that $\textbf{(H)}$ holds, then
	the semiflow $\mathcal{U}(t)_{t\geq 0}$ has a unique connected global attractor $\mathcal{M}$ on $(\mathbb{L}^{2}(\Omega))^{s}_{+}$.
\end{theorem}
\subsection{ Finite fractal dimensional }
\noindent

The main focus of the previous subsection was centered on the presence of the global attractor $\mathcal{M}$ for the semiflow $(\mathcal{U}(t))_{t\geq 0}$ in the space $(\mathbb{L}^{2}(\Omega))^{s}$. In the realm of infinite-dimensional dynamical systems, determining the fractal dimension of the global attractor is of great importance. This significance arises from Mañé's Theorem, which states that if the global attractor has a finite fractal dimension, it becomes possible to reduce the infinite-dimensional dynamical system to a finite-dimensional counterpart existing within the attractor.
Our goal in this section is to establish the finite fractal dimension property of the global attractor $\mathcal{M}$ of the semiflow $(\mathcal{U}(t))_{t\geq 0}$.

From \cite{Efendiev}, we recall the following definition.
\begin{definition} The fractal dimension of $\mathcal{M}$ is defined by
	\begin{center}
		$\dim_{f}(\mathcal{M})=\limsup\limits_{\varepsilon\to 0}ln(N_{\varepsilon})/ln(1/\varepsilon)$,
	\end{center} 
where $N_{\varepsilon}$ is the minor $n\in\mathbb{N}$ such that there exists open balls $B_1,\ldots,B_n$ with radius $\varepsilon$ such that $\mathcal{M}\subset \bigcup\limits_{i=1}^{n}B_i$.
\end{definition}

To show that $\mathcal{M}$ has a finite fractal dimension, we employ the following Theorem.
\begin{theorem} \cite[Theorem 2.1]{Czaja} Let $X$, $Y$, and $Z$ be Banach spaces such that $Y$ is compactly embedded in $Z$ and let $\mathcal{B}$ be a bounded set of $X$, invariant under a map $\mathcal{R}$ such that $\mathcal{R}(\mathcal{B})=\mathcal{B}$. Assume that there exist a map $\mathcal{T}:\mathcal{B}\to Y$, $\delta<\frac{1}{2}$ ($\delta\geq 0$), $\mu>0$, and $k>0$ such that
	\begin{equation*}
	\Vert \mathcal{R}(u)-\mathcal{R}(v)\Vert_{X}\leq \delta \Vert u-v\Vert_{X}+\mu\Vert \mathcal{T}(u)-\mathcal{T}(v)\Vert_{Z},
	\end{equation*} 
	and
	\begin{equation*}
	\Vert \mathcal{T}(u)-\mathcal{T}(v)\Vert_{Y}\leq k\Vert u-v\Vert_{X},
	\end{equation*}
	for all $u,v\in \mathcal{B}$. Under the above assumptions, $\mathcal{B}$ is a precompact set in $X$ and it has a finite fractal dimension, more precisely, for any $\nu\in (0,\frac{1}{2}-\delta)$, 
	\begin{equation*}
	\dim_{f}(\mathcal{B})\leq \frac{\ln\left(N_{\frac{\nu}{k\mu}}^{Z}\left(B_Y(0,1)\right)\right)}{\ln\left(\frac{1}{2(\delta+\nu)}\right)},
	\end{equation*}
	where $N_{\frac{\nu}{k\mu}}^{Z}\left(B_Y(0,1)\right)$ is the minimal number of $\frac{\nu}{k\mu}$-balls in $Z$ need to cover $B_Y(0,1)$ the open unit ball in $Y$.
	\label{finit fractal}
\end{theorem}

The following Theorem is the main result in this subsection.  It shows that the attractor $\mathcal{M}$ has a finite fractal dimension.
\begin{theorem} Assume that $\textbf{(H)}$ holds, then the global attractor $\mathcal{M}$ of the semiflow $(\mathcal{U}(t))_{t\geq 0}$ has a finite fractal dimension.
\end{theorem}
\begin{proof} Let $u_0,v_0\in \mathcal{M}$, and $\lambda:=\lambda(\mathcal{M})>0$ such that $e^{-\lambda t}\leq \delta$ for $0<\delta<\frac{1}{2\sqrt{e}}$ and $t>0$.  Writing equation \eqref{Eq 1} in the equivalent form \eqref{equation 5}, implies that the semiflow starting from $u_0$ and $v_0$ satisfies
	\begin{eqnarray*}
		& &\Vert \mathcal{U}(t)u_0-\mathcal{U}(t)v_0\Vert\\  &\leq&  e^{-\lambda t}\Vert u_0-v_0\Vert+\displaystyle\int_{0}^{t}e^{-\lambda (t-s)}\Vert F_\lambda(s,\mathcal{U}(s)u_0)-F_\lambda(s,\mathcal{U}(s)v_0)\Vert \, ds,
	\end{eqnarray*} 
	where $F_\lambda(t,v)=F(t,v)+\lambda v$ for $v\in (\mathbb{L}^{2}(\Omega))^{s}_{+}$.
	Since $u_0$ and $v_0$ are in $\mathcal{M}$ which is invariant under $(\mathcal{U}(t))_{t\geq 0}$, it follows that there exists $C:=C(\mathcal{M})>0$ such that 
	\begin{equation*}
	\Vert F_\lambda(s,\mathcal{U}(s)u_0)-F_\lambda(s,\mathcal{U}(s)v_0)\Vert\leq (\lambda+C)\Vert \mathcal{U}(s)u_0-\mathcal{U}(s)v_0\Vert.
	\end{equation*}
	Therefore,
	\begin{eqnarray*}
		\Vert \mathcal{U}(t)u_0-\mathcal{U}(t)v_0\Vert & \leq & \delta \Vert u_0-v_0\Vert +\displaystyle\int_{0}^{t}w(t,s)\Vert \mathcal{U}(s)u_0-\mathcal{U}(s)v_0\Vert ds,
	\end{eqnarray*}
	where $w(t,s)=e^{-\lambda (t-s)} (\lambda+C)$ for $t\geq s\geq 0$.
	By Lemma \ref{lem 2.2}, we get the following
	\begin{eqnarray*}
		 \Vert \mathcal{U}(t)u_0-\mathcal{U}(t)v_0\Vert
		 \!\!\!\!& \leq &\!\!\!\! \delta \exp\left(\displaystyle\int_{0}^{t}\!\left[\lambda+C-(\lambda+C)e^{-\lambda \tau}\!\!\displaystyle\int_{0}^{\tau} \lambda e^{\lambda r}dr\right]\!d\tau\right) \Vert u_0-v_0\Vert\\
		\!\!\!\!&=&\!\!\!\! \delta \exp\left(\displaystyle\int_{0}^{t}\left[\lambda+C-(\lambda+C)e^{-\lambda \tau}\left(e^{\lambda \tau}-1\right)\right]d\tau\right) \Vert u_0-v_0\Vert\\
		\!\!\!\!&=&\!\!\!\! \delta \exp\left((\lambda+C)\displaystyle\int_{0}^{t}e^{-\lambda \tau}d\tau\right) \Vert u_0-v_0\Vert\\
		\!\!\!\!&=&\!\!\!\! \delta \exp\left(\dfrac{(\lambda+C)}{\lambda}\left[1-e^{-\lambda t}\right]\right) \Vert u_0-v_0\Vert.
	\end{eqnarray*}
	Let $t^{*}>0$ such that $(1-e^{-\lambda t^*})\leq \frac{\lambda}{2(\lambda+C)}$. Then,
	\begin{equation*}
	\Vert \mathcal{U}(t^{*})u_0-\mathcal{U}(t^{*})v_0\Vert  \leq  \delta\sqrt{e} \Vert u_0-v_0\Vert \hspace{0.1cm} \text{ for all } \hspace{0.1cm} u_0,v_0\in \mathcal{M}.
	\end{equation*}
	Since $0<\delta\sqrt{e}<\frac{1}{2}$,  by applying Theorem \ref{finit fractal} for $X=Z=(\mathbb{L}^{2}(\Omega))^{s}$, $Y=(H^{1}(\Omega))^{s}$, $\mathcal{R}=\mathcal{U}(t^*)$, $\mathcal{T}=0$, $\mathcal{B}=\mathcal{M}$, and $\mu=k=1$, we obtain that $\mathcal{M}$ has a finite fractal dimension, moreover, for any $\nu\in (0,\frac{1}{2}-\delta\sqrt{e})$, 
	\begin{equation*}
	\dim_{f}(\mathcal{M})\leq \frac{\ln\left(N_{\nu}\left(B_{H^1}(0,1)\right)\right)}{\ln\left(\frac{1}{2(\delta\sqrt{e}+\nu)}\right)},
	\end{equation*}
	where $N_{\nu}\left(B_{H^1}(0,1)\right)$ is the minimal number of ${\nu}$-balls in $(\mathbb{L}^{2}(\Omega))^{s}$ needed to cover $B_{H^1}(0,1)$ the open unit ball in $(H^{1}(\Omega))^{s}$.
\end{proof}

\section{Analytical solution of an Advection-Diffusion Equation in 2-D using separation of variables} \label{sec 4}
\noindent

Finding an analytical solution to the Advection-Diffusion-Reaction (ADR) equation, with constant advection and diffusion coefficients, Dirichlet boundary conditions, and an initial condition, which considers the reaction in its general form, is often infeasible. This is due to the non-linearity added by the reaction term $R$, the sum of reactions that are products of concentrations. The analytical solution exists for certain forms of $R$ that can be simplified. In the literature, the Laplace Transform method is used to solve the ADR equation, with constant advection and diffusion coefficients, Dirichlet boundary conditions, and an initial condition, with a constant first-order reaction and a source function \cite{KIM}. Green’s function is used when partial differential equations are linear, and the integral involving the reaction term in the Green function approach converges \cite{Parhizi}. The analytical solution of the Advection-Diffusion equation, with constant advection and diffusion coefficients, Dirichlet boundary conditions, and an initial condition, using the separation of variables approach while considering there is no reaction between species, is revisited in this section. This analytical solution is one of a subset of ADR equations, with constant advection and diffusion coefficients, Dirichlet boundary conditions, and an initial condition.
\subsection{Analytical solution in 2-D}
\noindent

The solution of the following two-dimensional advection-diffusion equation with constant advection and diffusion coefficients:

\begin{equation}
\left\{\begin{array}{l}
\dfrac{\partial c}{\partial t}=k \dfrac{\partial^2 c}{\partial x^2}+k \dfrac{\partial^2 c}{\partial y^2}-u \dfrac{\partial c}{\partial x}-u\dfrac{\partial c}{\partial y},\quad \Omega=[0,1]\times[0,1],\\[10pt]
c_{\mid \partial \Omega}=0, \quad\\[10pt]
c(0,x, y,)=f(x, y). \quad
\end{array}\right.
\label{equation 6}
\end{equation}
is discussed in \cite{Buckman}. The solution is of the separable form
\begin{equation}
c(t,x, y)=X(x) Y(y) T(t).
\label{eq_sep}
\end{equation}
Substituting into equation \eqref{equation 6} and dropping $x, y, t$ terms
$$
X Y T^{\prime}=\left(k X^{\prime \prime}-u X^{\prime}\right) Y T+X\left(k Y^{\prime \prime}-u Y\right) T.
$$
Rearranging leads to
$$
\frac{T^{\prime}}{T}=\frac{k X^{\prime \prime}-u X^{\prime}}{X}+\frac{k Y^{\prime \prime}-u Y^{\prime}}{Y}.
$$
The variables $t$, $x$ and $y$ are supposed independent, then each function variable in the above equation is a constant. Meaning this equation is only possible if each fraction is equal to a number $\lambda$, $\alpha$, and $\beta$ where
$\lambda=\alpha+\beta$. From symmetry, the solutions for the $x$-dependent differential equation and $y$-dependent differential equation will be the same with different constants (eigenvalues). The eigenproblem gets reduced to two first order differential equations with

\begin{gather}
T^{\prime}(t)=\lambda T(t), \quad T(0)=1, \label{equation 7}\\[10pt]
k X^{\prime \prime}(x)-u X^{\prime}(x)=\alpha X(x). \label{equation 8}
\end{gather}
Solving equation \eqref{equation 7}, we get
\begin{equation}
T(t)=e^{\lambda t}. \label{equation t}
\end{equation}
Below the solution of equation \eqref{equation 8} is worked out, and the characteristic equation for equation \eqref{equation 8} is presented:
\begin{gather}
k r^2-u r-\alpha=0;  \hspace{0.1cm} \mu=u^2+4 k \alpha; \hspace{0.1cm}
r_1=\frac{u+\sqrt{u^2+4 k \alpha}}{2 k} ; \hspace{0.1cm} r_2=\frac{u-\sqrt{u^2+4 k \alpha}}{2 k}.
\label {equation 9}
\end{gather}
The solution will depend on the discriminant \eqref{equation 9} of the characteristic equation. The three possible solutions are: 
\begin{gather*}
X(x)=\left\{\begin{array}{l}
E e^{r_1 x}+F e^{r_2 x}, \quad \mu>0 \\
C e^{r_1 x}+D, \quad \mu=0 \\
e^{\frac{u}{2k} x}\left[A \sin \left(\frac{\sqrt{-u^2-4 k \alpha}}{2 k} x\right)+B \cos \left(\frac{\sqrt{-u^2-4 k \alpha}}{2 k} x\right)\right], \quad \mu<0.
\end{array}\right.
\end{gather*}
However the only equation, from which we can construct a non-trivial solution, is when $\mu<0$ and thus $r$ is imaginary:
\begin{equation*}
X(x)=e^{\frac{u}{2 k} x}\left[A \sin \left(\frac{\sqrt{-u^2-4 k \alpha}}{2 k} x\right)+B \cos \left(\frac{\sqrt{-u^2-4 k \alpha}}{2 k} x\right)\right].
\end{equation*}
The following steps determine the solution that satisfies the Dirichlet boundary conditions and then the initial condition. We have,
\begin{align*}
& c(t,0,y)=0\implies X(0)Y(y)T(t)=0,\\
& c(t,1,y)=0\implies X(1)Y(y)T(t)=0.\\ \\
\text{Hence,} \qquad & X(0)=0 \text{ and } X(1)=0,\\
& X(0) = 0 \implies B=0 \hspace{0.1cm}\text{ and }\hspace{0.1cm} X(1)=0 \implies\sin\left(\frac{\sqrt{-u^2-4 k \alpha}}{2 k}\right)=0.\\
\text{Thus,}  \qquad & \frac{\sqrt{-u^2-4 k \alpha}}{2 k}=m \pi \text { for } m=1,2,3, \ldots.
\end{align*}
The solution is reduced to
\begin{align*}
X_m(x)=A_m e^{\frac{u}{2 k} x} \sin (m \pi x) \text { for } m=1,2,3 \ldots,\hspace{0.1cm} \text{ with }\hspace{0.1cm} \alpha_m=-k m^2 \pi^2-\frac{u^2}{4 k}.
\end{align*}
Based on the superposition principle, any linear combination of $X(x)$ is a solution of a linear homogeneous problem. Similarly, the eigenfunctions $X_m(x)$ and their corresponding eigenvalues will have the same form for $Y(y)$. The product of the solutions for the $x$-dependent and $y$-dependent differential equations will take the form of a Fourier sine series, satisfying Dirichlet boundary conditions:

$$
X(x) Y(y)=\sum_{m=1}^{\infty} \sum_{n=1}^{\infty} A_{m, n} e^{\frac{u}{2 k} x} e^{\frac{u}{2 k} y} \sin (m \pi x) \sin (n \pi y). \notag
$$
Now, remains finding the Fourier coefficient $A_{m, n}$ such that this expression holds true under the initial condition. Using the orthogonality properties of the eigenfunctions and by combining equations \eqref{eq_sep}, \eqref{equation t} and the initial condition, we get\\
$$
c(0,x, y)=f(x, y)=X(x) Y(y) T(0)=X(x) Y(y),
$$
which implies that
$$
 f(x, y)=X(x) Y(y)=\sum_{n=1}^{\infty} \sum_{n=1}^{\infty} A_{m, n} e^{\frac{u}{2k} x} e^{\frac{u}{2 k} y} \sin (m \pi x) \sin (n \pi y).
$$
Based on Fubini's Theorem, under the condition of absolute convergence, integrating yields
\begin{align*}
&\int_0^1 \int_0^1 f(x, y) d x d y=\int_0^1 \int_0^1 \sum_{m=1}^{\infty} \sum_{n=1}^{\infty} A_{m, n} e^{\frac{u}{2 k} x} e^{\frac{u}{2k} y} \sin (m \pi x) \sin (n \pi y) d x d y,
\end{align*}
which implies that
\begin{align}
&\int_0^1 \int_0^1 f(x, y) e^{-\frac{u}{2 k} x} e^{-\frac{u}{2 k} y} \sin (l \pi x) \sin (k \pi y) d x d y= \notag\\
&\sum_{m=1}^{\infty} \sum_{n=1}^{\infty} \int_0^1 \int_0^1 A_{m,n} \sin (m \pi x) \sin (n \pi y) \sin (l \pi x) \sin (k \pi y) d x d y,
\label{eq_multiply}
\end{align}
\begin{gather*}
\left\{\begin{array}{l}
\text{when } k,l \neq m,n \implies \eqref{eq_multiply}=0\\
\text{when }k,l= m, n \implies \eqref{eq_multiply}=\frac{1}{4} A_{k, l}.
\end{array}\right.
\end{gather*}
The previous implication occurs when $k,l= m, n$ and stems from:
\begin{align*}
&\sum_{m=1}^{\infty} \sum_{n=1}^{\infty} \int_0^1 \int_0^1 A_{m,n} \sin (m \pi x) \sin (n \pi y) \sin (l \pi x) \sin (k \pi y) d x d y\notag\\
&=\frac{1}{4} A_{k, l} \int_0^1 \int_0^1 2 \sin ^2\left(k \pi x\right) 2 \sin ^2(l \pi y) d x d y\\
& =\frac{1}{4} A_{k, l} \int_0^1(1-\cos (2 k \pi x)) d x \int_0^1(1-\cos (2 k \pi y)) d y \\
& =\frac{1}{4} A_{k, l} \left[\left(x-\left.\frac{\sin (2 k \pi x)}{2 k \pi}\right|_0 ^1\right)\left(y-\left.\frac{\sin (2 l \pi y)}{2 k \pi}\right|_0 ^1\right)\right] \\
& =\frac{1}{4} A_{k,l}.
\end{align*}
Therefore \eqref{eq_multiply} reduces to:
\begin{equation}
A_{m, n}=4 \int_0^1 \int_0^1 f(x, y) e^{-\frac{u}{2 k} x} e^{-\frac{u}{2 k} y} \sin (m \pi x) \sin (n \pi y) d x d y, \quad m, n=1,2,3 \ldots \label{equation 10}
\end{equation}
Now back to the time-dependent solution:
\begin{align*}
T(t)=e^{\lambda t},\quad \lambda_{m,n}=\alpha_m+\beta_n=-k m^2 \pi^2-\frac{u^2}{4 k}-k n^2 \pi^2-\frac{u^2}{4 k}.
\end{align*}
Combining the time-dependent solution with the spatial solution we get the final solution to the advection-diffusion equation:
\begin{equation*}
c(t,x, y) =\sum_{m=1}^{\infty} \sum_{n=1}^{\infty} A_{m, n} e^{\left(-k(m^2+n^2)\pi^2-\frac{u^2}{2 k}\right) t} e^{\frac{u}{2 k}(x+y)} \sin (m \pi x) \sin (n \pi y),
\end{equation*}
where $A_{m, n}$ is given by \eqref{equation 10}.

\subsection{Analytical Example}
\noindent

We will specify the following advection-diffusion equation:
\begin{equation*}
\left\{\begin{array}{l}
\dfrac{\partial c}{\partial t}=0.5 \dfrac{\partial^2 c}{\partial x^2}+0.5 \dfrac{\partial^2 c}{\partial y^2}-5 \dfrac{\partial c}{\partial x}-5 \dfrac{\partial c}{\partial y}, \\[10pt]
c(0,x, y)=f(x, y)=\sin (\pi x) \sin (\pi y).
\end{array}\right.
\end{equation*}
Applying these new advection-diffusion coefficients and initial conditions, the solutions are:
$$
\begin{aligned}
c(t,x,y) =\sum_{m=1}^{\infty} \sum_{n=1}^{\infty} A_{m, n} e^{\left(-0.5\pi^2( m^2 + n^2)-25\right) t} e^{5(x+y)} \sin (m \pi x) \sin (n \pi y),
\end{aligned}
$$
where 
\begin{center}
	$A_{m, n} =4 \displaystyle\int_0^1 \displaystyle\int_0^1 e^{-5(x+y)} \sin (\pi x) \sin (\pi y) \sin (m \pi x) \sin (n \pi y) d x d y$, \quad $m,n\geq 1$.
\end{center}
Graphed in Figure \ref{fig:analytic_sol} are the solutions found using initial conditions and coefficients.
The coefficients of the Fourier sine series can be calculated by computer, in this example taking the first 40 terms. We graphed these solutions using Python 3.10.5 at different times stamps. The analytical solutions diffuse over time and move away from the origin. In the next section, we will discuss a numerical simulation of this same example, and compare the results (convergence and graphs) to these analytical solutions.
\begin{figure}[H]
	\centering
	\includegraphics[scale=0.3]{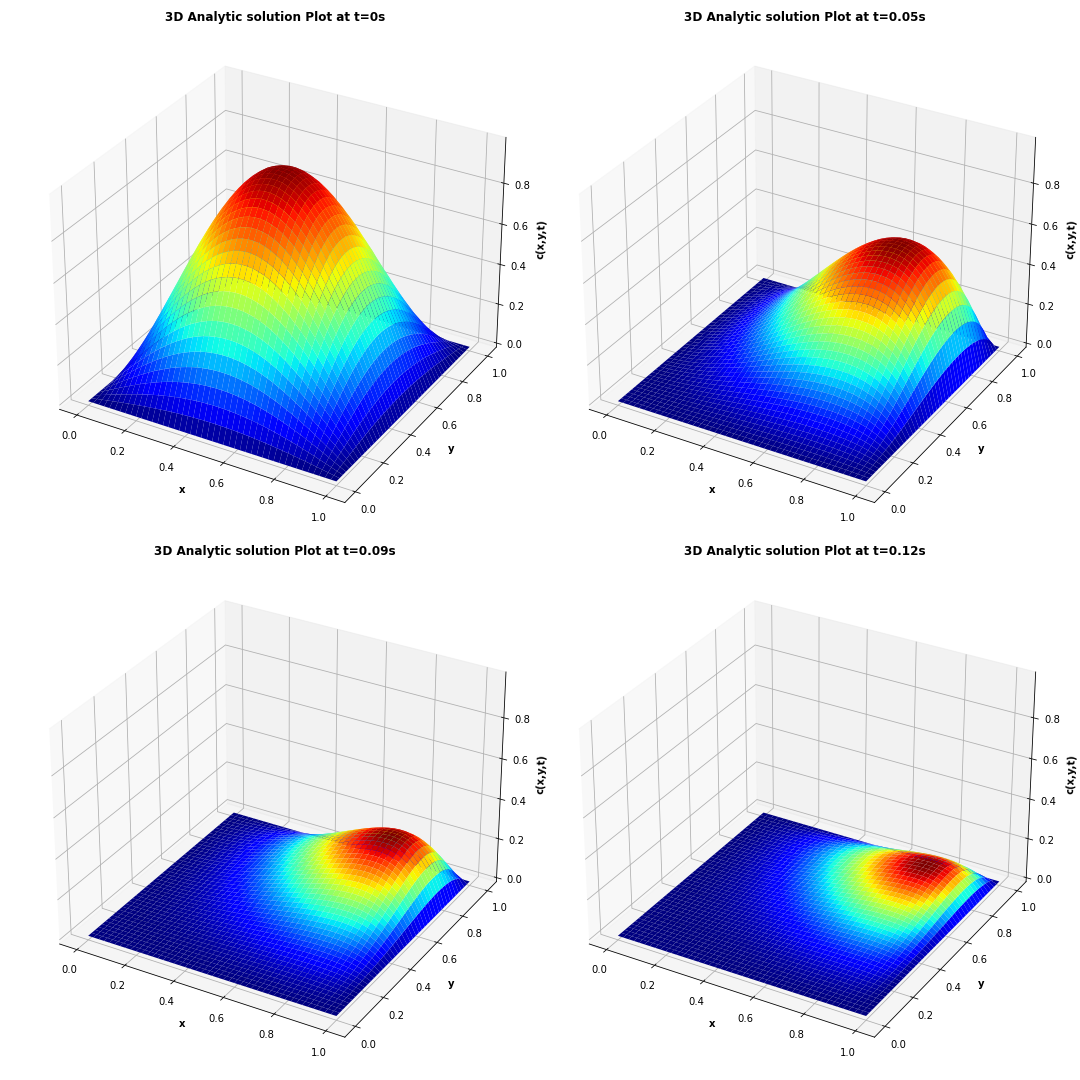}
	\caption{Analytical solutions at times 0, 0.05, 0.09, 0.12, with $u=5$ and $k=0.5$.}
	\label{fig:analytic_sol}
\end{figure}

\section{Numerical Simulations with Finite Difference Methods} \label{sec 5}
\noindent

Numerical methods, such as finite difference, can provide accurate and efficient approximations to the analytic solution, allowing for the study of complex systems and realistic scenarios that do not admit simple analytical solutions. In this section, we will solve numerically the example presented above, and present a new example including the reaction between Ozone, $\mathrm{NO}$ and $\mathrm{NO}_2$ in clean atmosphere. 
\subsection{Numerical simulations in 2-D using explicit centered difference method}
\noindent

Considering the same advection-diffusion equation and initial conditions from the analytic example above, we demonstrate the discretization using the explicit centered difference method. The time derivative is approximated using a Forward Difference scheme, while the advection and diffusion terms are approximated using a Central Difference scheme. Here, $i$ and $j$ represent the spatial indices, and $n$ represents the temporal index. The solution at the grid point $\left(x_i, y_j\right)$ at time $n$ is denoted by $c_{i,j}^n$:
\begin{equation*}
\begin{aligned}
\frac{c_{i, j}^{n+1}-c_{i, j}^n}{\delta t}= & -u \frac{c_{i+1, j}^n-c_{i-1, j}^n}{2\delta x} -u \frac{c_{i+1, j}^n-c_{i, j-1}^n}{2\delta y}\\
& + k \frac{c_{i+1, j}^n-2 c_{i, j}^n+c_{i-1, j}^n}{\delta x^2} +k \frac{c_{i, j+1}^n-2 c_{i, j}^n+c_{i, j-1}^n}{\delta y^2}.
\end{aligned}
\end{equation*}

\subsubsection{Stability}
\noindent

Consider a uniform grid with constant advection-diffusion coefficients. We introduce the Diffusion number and the Cell Peclet number as follows:
\[
\begin{gathered}
R_x = k \frac{\delta t}{\delta x^2}= R_y = k \frac{\delta t}{\delta y^2}, \quad
P_x = \frac{u \delta x}{2k}, \quad \text{and} \quad P_y = \frac{u \delta y}{2k}.
\end{gathered}\]

The finite difference scheme rearranged gives:
\begin{equation*}
\begin{aligned}
c_{i, j}^{n+1}=c_{i, j}^n+ & R_x\left(c_{i+1, j}^n-2 c_{i, j}^n+c_{i-1, j}^n\right)+R_y\left(c_{i, j+1}^n-2 c_{i, j}^n+c_{i, j-1}^n\right)\\
&-P_x R_x\left(c_{i+1, j}^n-c_{i-1, j}^n\right) -P_y R_y\left(c_{i, j+1}^n-c_{i, j-1}^n\right).
\end{aligned}
\end{equation*}
Grouping like terms,
\begin{equation}
\begin{aligned}
c_{i, j}^{n+1}= & \left(1-2 R_x -2 R_y\right) c_{i, j}^n+ \left(R_x-P_xR_x\right) c_{i+1, j}^n +\left(R_x+P_xR_x\right) c_{i-1, j}^n \\
& + \left(R_y-P_yR_y\right) c_{i, j+1}^n+\left(R_y + P_yR_y \right) c_{i, j-1}^j.
\end{aligned}
\label{equ_8}
\end{equation}
The solution will be stable if the coefficients are all positive. The following constraints on $R_x$ and $P_x$ rise:
$$
1-2 R_x-2 R_y=1-4 R_x>0 \quad or \quad R_x<\frac{1}{4} \quad and \quad
R_x-P_x R_x>0 \quad or \quad P_x<1.
$$
\subsubsection{Convergence}
\noindent

Next, the convergence of the explicit centered method is analyzed using a Taylor expansion for each term:
\begin{align*}
& c_{i, j}^{n+1}=c_{i, j}^n+\left(\frac{\partial c}{\partial t}\right)_{i, j}^n \delta t+\left(\frac{\partial^2 c}{\partial t^2}\right)_{i, j}^n \frac{\delta \mathrm{t}^2}{2 !}+\ldots, \\
& c_{i+1, j}^j=c_{i, j}^n+\left(\frac{\partial c}{\partial x}\right)_{i,j}^n \delta x+\left(\frac{\partial^2 c}{\partial x^2}\right)_{i, j}^n \frac{\delta x^2}{2 !}+\left(\frac{\partial^3 c}{\partial x^3}\right)_{i, j}^n \frac{\delta x^3}{3 !} \ldots, \\
& c_{i-1, j}^n=c_{i, j}^n-\left(\frac{\partial c}{\partial x}\right)_{i, j}^n \delta x+\left(\frac{\partial^2 c}{\partial x^2}\right)_{i, j}^n \frac{\delta x^2}{2 !}-\left(\frac{\partial^3 c}{\partial x^3}\right)_{i, j}^n \frac{\delta x^3}{3 !} \ldots, \\
& c_{i, j+1}^n=c_{i, j}^n+\left(\frac{\partial c}{\partial y}\right)_{i, j}^n \delta y+\left(\frac{\partial^2 c}{\partial y^2}\right)_{i, j}^n \frac{\delta y^2}{2 !}+\left(\frac{\partial^3 c}{\partial y^3}\right)_{i, j}^n \frac{\delta y^3}{3 !} \ldots, \\
& c_{i, j-1}^n=c_{i, j}^n-\left(\frac{\partial c}{\partial y}\right)_{i, j}^n \delta y+\left(\frac{\partial^2 c}{\partial y^2}\right)_{i, j}^n \frac{\delta y^2}{2 !}-\left(\frac{\partial^3 c}{\partial y^3}\right)_{i, j}^n \frac{\delta y^3}{3 !} \ldots.
\end{align*}
Plugging into \eqref{equ_8}, simplifying and dropping the subscripts
$$
\begin{gathered}
\frac{\partial c}{\partial t}=k \frac{\partial^2 c}{\partial x^2}+\mathrm{k} \frac{\partial^2 c}{\partial y^2}-u \frac{\partial c}{\partial t}-u \frac{\partial c}{\partial x}-\frac{\delta t}{2 !}\left(\frac{\partial^2 c}{\partial t^2}\right)-\frac{u \delta x^2}{3 !} \frac{\partial^3 c}{\partial x^3} -\frac{u \delta y^2}{3 !} \frac{\partial^3 c}{\partial y^3}.
\end{gathered}
$$
The leading error terms are
$$
E=-\frac{\delta t}{2!}\left(\frac{\partial^2 c}{\partial t^2}\right)-\frac{u \delta x^2}{3 !} \frac{\partial^3 c}{\partial x^3}-\frac{u \delta y^2}{3 !}\frac{\partial^3 c}{\partial y^3}.
$$
So the error is second order in space and first order in time Error $\sim \delta t, \delta x^2, \delta y^2$.
In addition, as $\delta t, \delta x^2, \delta y^2 \rightarrow 0$ the numerical solution converges to the analytical solution showing consistency.
\subsubsection{Numerical Example}
\noindent

Using the explicit centered method, relatively small step sizes are chosen in the $x$, $y$, and $t$ directions, with $nx=46$ and $ny = 46$ representing the number of grid points in the $x$- and $y$-directions. The final resolution is $\delta x = \delta y = \frac{L}{nx-1} = \frac{L}{ny-1} = 0.02222$, and $\delta t = 0.0001$, which are chosen to verify the stability conditions: $R_x = 0.10125 < \frac{1}{4}$ and $P_x = 0.1112 < 1$. In Figure \ref{fig:numerical_sol}, the numerical solutions diffuse over time and move away from the origin. As shown in the previous section, the numerical solution converges to the analytical solution, as evidenced by the similar graphs, which also show the properties of diffusion and dispersion.

\subsubsection{Error Analysis}
\noindent

Figure \ref{fig:error_plot} shows that the explicit centered numerical scheme accurately approximates the analytical solution, given the coefficients that satisfy the stability conditions of the scheme.
\subsection{Numerical simulations in 3-D using the Explicit Upwind Scheme}
\noindent

Ozone is produced through two reactions involving four species in a clean atmosphere. However, in the presence of pollution, ozone primarily forms as a result of a chemical reaction between $\mathrm{NO}_2$ and $\mathrm{VOC}$, which accounts for 99\% of its formation. The example presented below does not aim to model the reactions in detail but instead seeks to demonstrate the mathematical and physical principles by carrying out simulations in a clean atmosphere.
Consider the following three-dimensional Advection-Diffusion-Reaction equation:
\begin{equation*}
\frac{\partial c}{\partial t} + u_x \frac{\partial c}{\partial x} + u_y \frac{\partial c}{\partial y} + u_z \frac{\partial c}{\partial z} = k_x \frac{\partial^2 c}{\partial x^2} + k_y \frac{\partial^2 c}{\partial y^2} + k_z \frac{\partial^2 c}{\partial z^2} + R(c).
\end{equation*}
Discretization is performed using the explicit First-Order Upwind Scheme. The time derivative is approximated using the Forward Difference method, the advection term is approximated using a One-Sided Difference method (for $u > 0$), and the diffusion term is approximated using the Central Difference method.
\begin{figure}[H]
	\centering
	\includegraphics[scale=0.3]{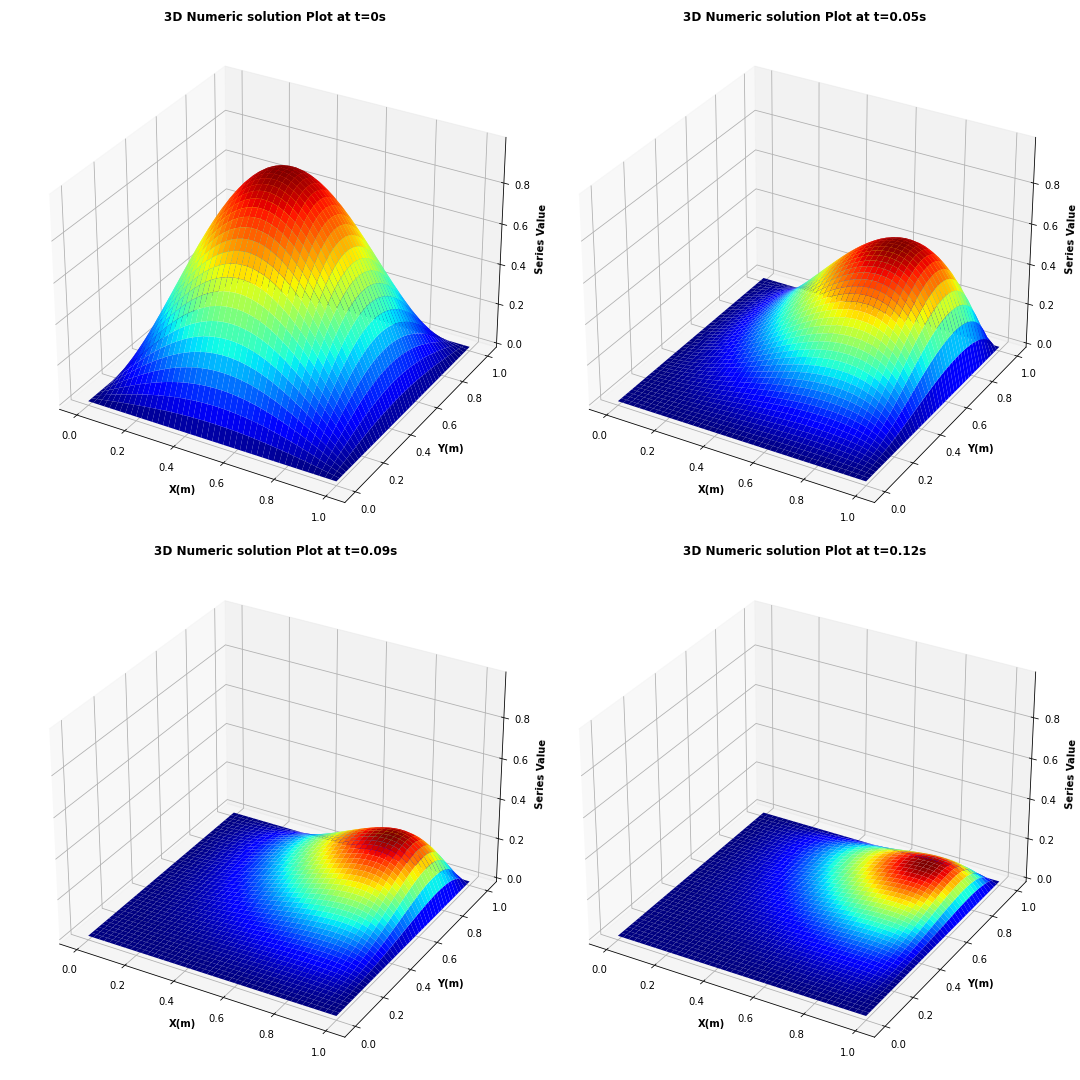}
	\caption{Numerical solutions at times 0, 0.05, 0.09, 0.12, with $u=5$ and $k=0.5$.}
	\label{fig:numerical_sol}
\end{figure}
\noindent
 Here, $i$, $j$, and $k$ represent the spatial indices, and $n$ represents the temporal index. $c_{i,j,k}^n$ denotes the concentration of each pollutant at grid point $\left(x_i, y_j, z_k\right)$ and time $n$.

\begin{eqnarray*}
	 c_{i, j, k}^{n+1} & \!\!\!=\!\!\! & c_{i, j, k}^{n}+\delta t \left[ -u_x \frac{c_{i, j, k}^n -c_{i-1, j, k}^n}{\delta x}-u_y \frac{c_{i, j, k}^n -c_{i, j-1, k}^n}{\delta y}-u_x \frac{c_{i, j, k}^n -c_{i, j, k-1}^n}{\delta z}\right.\\[10pt]
	& & \left. + k_x \frac{c_{i+1, j, k}^n -2c_{i, j, k}^n+ c_{i-1, j, k}^n}{\delta x^2} +k_y \frac{c_{i, j+1, k}^n -2c_{i, j, k}^n+ c_{i, j-1, k}^n}{\delta y^2}\right.\\[10pt]
	& & \left. +k_z \frac{c_{i, j, k+1}^n-2 c_{i, j, k}^n+c_{i, j, k-1}^n}{\delta z^2} +R(c_{i, j, k}^n) \right].
\end{eqnarray*}
Considering 2 reactions between  $\mathrm{NO},\mathrm{NO}_2, \mathrm{O}_3$ and $\mathrm{O}_2$:
\[
\mathrm{NO}_2+\mathrm{O}_2+h v \stackrel{\mathrm{k}_1}{\longrightarrow} \mathrm{NO}+\mathrm{O}_3 \quad  \text{and} \quad
\mathrm{NO}+\mathrm{O}_3 \stackrel{\mathrm{k}_2}{\longrightarrow} \mathrm{O}_2+\mathrm{NO}_2.
\]
These reactions are basic to tropospheric air pollution models. The first reaction is photochemical and states that $\mathrm{NO}$ and $\mathrm{O}_3$ are formed from the photo-dissociation $hv$ caused by solar radiation of $\mathrm{NO}_2$ and $\mathrm{O}_2$. This depends on the time of the day and therefore $k_l = k_l(t)$. Let the concentrations of the pollutants be $c_1=[\mathrm{N O}], c_2=\left[\mathrm{NO}_2\right]$ and $c_3=\left[\mathrm{O}_3\right]$. $\left[O_2\right]$ is constant, and consider a constant source term $\sigma_2$ at the cell $(1,1,1)$ simulating emissions of $\mathrm{NO}$. The speed of reactions are:
\[
\begin{aligned}
& R_1(t,c_1(t))=f\left(t, c_1(t)\right)=\frac{\partial c_1(t)}{\partial t}=k_1(t) c_2(t)-k_2 c_1(t) c_3(t)+\sigma_2, \\
& R_2(t,c_2(t))=f\left(t, c_2(t)\right)=\frac{\partial c_2(t)}{\partial t}=k_2 c_1(t) c_3(t)-k_1(t) c_2(t), \\
& R_3(t,c_3(t))=f\left(t, c_3(t)\right)=\frac{\partial c_3(t)}{\partial t}=k_1(t) c_2(t)-k_2 c_1(t) c_3(t).
\end{aligned}
\]
The system of ODEs can be written in the following form:
\[
R(t,c(t))=c^{\prime}(t)=S g(t, c(t)), \quad c(0) \text { is given },
\]
where $S=\left(r_{i j}-l_{i j}\right)$ is an $s \times r$ matrix, of the loss and gain of the number of molecules $U_i$ in the $j^{th}$ reaction, where $s$ is the number of species, $r$ is the number of reactions, $g(t, c)=\left(g_j(t, c)\right) \in \mathbb{R}^r$,
\[g(t, c)=\left(\begin{array}{c}
	k_1(t) c_2(t) \\
	k_2 c_1(t) c_3(t) \\
	\end{array}\right), \quad S=\left(\begin{array}{rrr}
	1 & -1  \\
	-1 & 1 \\
	1 & -1
	\end{array}\right).\]
\begin{figure}[htbp]
	\centering
	\includegraphics[scale=0.85]{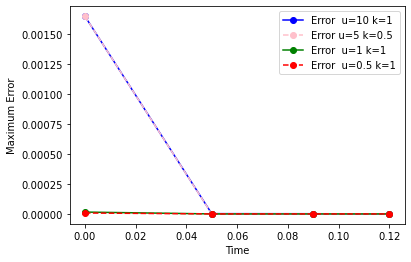}
	\caption{Maximum error between Analytical and Numerical solutions at times 0, 0.05, 0.09, 0.12, with $u=5$, $k=0.5$, and $m=n=40$.}
	\label{fig:error_plot}
\end{figure}
\noindent
\hspace{-0.12cm}Let \( c(0,(1,1,1)) = \left( 1.3 \times 10^8, 5.0 \times 10^{11}, 8.0 \times 10^{11} \right)^T \) be the initial values of the concentrations of the three respective pollutants, in molecules per \( \text{cm}^3 \), at the initial cell \( (1,1,1) \) for each time-step \( \delta t = 1 \, \text{s} \). The reaction coefficient \( \sigma_2 = 10^6 \) is given, and the reaction coefficients are as follows:  
\[
k_2 = 10^{-16}, \quad k_1(t) = \begin{cases} 
10^{-5} e^{7 \sec (t)} & \text{in daytime}, \\
10^{-40} & \text{during nighttime},
\end{cases}
\]  
where
\[
\sec (t) = \left( \sin \left( \frac{1}{16} \pi \left( \bar{t}_h - 4 \right) \right) \right)^{0.2}, \quad \bar{t}_h = t_h - 24 \left\lfloor \frac{t_h}{24} \right\rfloor, \quad t_h = \frac{t}{3600}.
\]  
Daytime is defined between 4 am and 8 pm, and \( \left\lfloor \frac{t_h}{24} \right\rfloor \) represents the largest integer less than or equal to \( \frac{t_h}{24} \). The function \( \sec(t) \) is periodic with a period of 24 hours but is defined only during daytime. The maximum value for \( k_1 \) occurs at noon, with \( k_1 \approx 0.01 \). The solution will be stable with the following constraint:
\[
2 R_x+2 R_y+2R_z+P_xR_x+P_yR_y+P_zR_z<1,
\]
where
\[
R_x = \frac{k_x \delta t}{\delta x^2}, \quad R_y = \frac{k_y \delta t}{\delta y^2}, \quad R_z = \frac{k_z \delta t}{\delta z^2},
\]
and
\[
P_x = \frac{u_x \delta x}{k}, \quad P_y = \frac{u_y \delta y}{k}, \quad P_z = \frac{u_z \delta z}{k}.
\]
The computational parameters for our simulations are as follows:
\[
\begin{aligned}
& \delta t=1s,\\
& \delta x=\delta y=\delta z=10m.\\
& \text{Initial conditions}=c(0,(1,1,1))=\left(1.3 \times 10^8, 5.0 \times 10^{11}, 8.0 \times 10^{11}\right)^T.\\
& \text{Boundary conditions: }c=(0,0,0)^T \text{at $x$, $y$ or $z = 0 m$ and at $x, y$ or $z = 1000 m$},\\
& u_x=u_y=u_z=1m/s,\\
& k_x=k_y=k_z=0.00002 m^2/s.\\
\end{aligned}
\]
For \( \delta x = \delta y = \delta z = 10 \, \text{m} \), each cell's volume is \( 10 \, \text{m}^3 \). We consider a time step of \( 1 \, \text{s} \) due to the stability constraint of the First-Order Upwind Scheme, known as the CFL (Courant-Friedrichs-Lewy) condition:  
\[
\max \left( \frac{u_x \delta t}{\delta x}, \frac{u_y \delta t}{\delta y}, \frac{u_z \delta t}{\delta z} \right) \leq \alpha,
\]
where \( \alpha \) is a certain value less than $1$. Dirichlet boundary conditions are applied at \( x, y, z = 0 \) and \( 1000 \, \text{m} \). Figures \ref{fig:example}-\ref{fig:example3} display horizontal slice plots at \( z=1 \) showing the distribution of pollutant concentrations of \( \mathrm{NO}_2 \), \( \mathrm{O}_3 \), and \( \mathrm{NO} \) at times \( 0 \, \text{s}, 20 \, \text{s}, 50 \, \text{s}, 100 \, \text{s}, 200 \, \text{s}, 300 \, \text{s}, 400 \, \text{s}, 500 \, \text{s}, \) and \( 600 \, \text{s} \), with \( u_x=u_y=u_z = 1 \, \text{m/s} \) and \( k_x=k_y=k_z = 0.00002 \, \text{m}^2/\text{s} \). The plots reveal the concentration distribution of the dispersed gases across the horizontal plane at \( z=1 \), highlighting areas of high and low concentration. These plots illustrate regions where the gas is more or less prevalent. The multiple slices at different time steps demonstrate how gas concentrations evolve over time due to advection, diffusion, and reactions. With \( u_x = u_y = u_z = 1 \, \text{m/s} \), air currents significantly impact gas transport. The concentration gradients show how gases are transported from regions of higher concentration to lower concentration due to advection. Diffusion is evident in the gradual spreading of gas concentrations away from their sources. In Figure \ref{fig:example}, the maximum concentration of \( \mathrm{NO}_2 \) decreases from \( 5 \times 10^{18} \, \text{molecules}/10 \, \text{m}^3 \) at \( t = 0 \, \text{s} \) to \( 74.385 \, \text{mol.}/10 \, \text{m}^3 \) at \( 200 \, \text{s} \), \( 0.078 \, \text{mol.}/10 \, \text{m}^3 \) at \( 300 \, \text{s} \), and \( 1.57 \times 10^{-19} \, \text{mol.}/10 \, \text{m}^3 \) at \( 600 \, \text{s} \). In Figure \ref{fig:example2}, the maximum concentration of \( \mathrm{O}_3 \) decreases from \( 8 \times 10^{18} \, \text{mol.}/10 \, \text{m}^3 \) at \( t = 0 \, \text{s} \) to \( 119.017 \, \text{mol.}/10 \, \text{m}^3 \) at \( 200 \, \text{s} \), \( 0.126 \, \text{mol.}/10 \, \text{m}^3 \) at \( 300 \, \text{s} \), and \( 2.52 \times 10^{-19} \, \text{mol.}/10 \, \text{m}^3 \) at \( 600 \, \text{s} \). In Figure \ref{fig:example3}, the maximum concentration of \( \mathrm{NO} \) starts at \( 1.3 \times 10^{15} \, \text{mol.}/10 \, \text{m}^3 \) at \( t = 0 \, \text{s} \) and decreases to \( 3.3 \times 10^{13} \, \text{mol.}/10 \, \text{m}^3 \) at \( 200 \, \text{s} \), remaining approximately constant at \( 300 \, \text{s} \) and \( 600 \, \text{s} \). In Figure \ref{fig:example3}, a constant source of \( \mathrm{NO} \) is present at cell \( (1, 1, 1) \). This explains why the concentration plot exhibits a relatively uniform concentration, with distinct variations around the source point due to interactions and dispersion processes. Close to the source point, higher concentrations are observed due to the immediate release of \( \mathrm{NO} \).
\begin{landscape}
	\begin{figure}[H]
		\centering
		\includegraphics[width=1.213\textwidth]{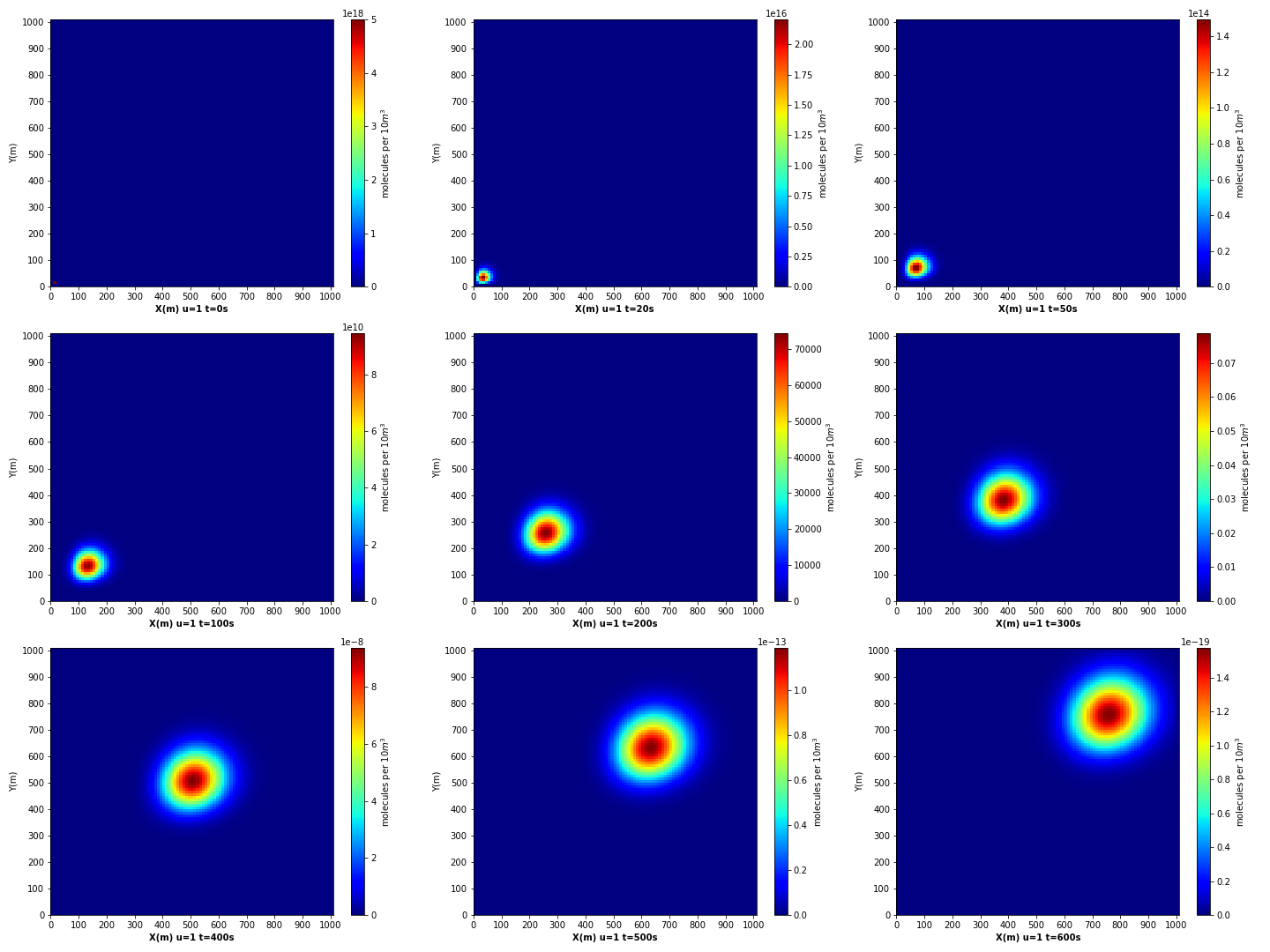}
		\caption{Horizontal slice plots at \( z = 1 \) showing \( \mathrm{NO}_2 \) concentrations at times \( 0 \, \text{s}, 20 \, \text{s}, 50 \, \text{s}, 100 \, \text{s}, 200 \, \text{s}, 300 \, \text{s}, 400 \, \text{s}, 500 \, \text{s}, \) and \( 600 \, \text{s} \), with \( u = 1 \, \text{m/s} \) and \( k = 0.00002 \, \text{m}^2/\text{s} \).}
		\label{fig:example}
	\end{figure}
	
	\begin{figure}[H]
		\centering
		\includegraphics[width=1.213\textwidth]{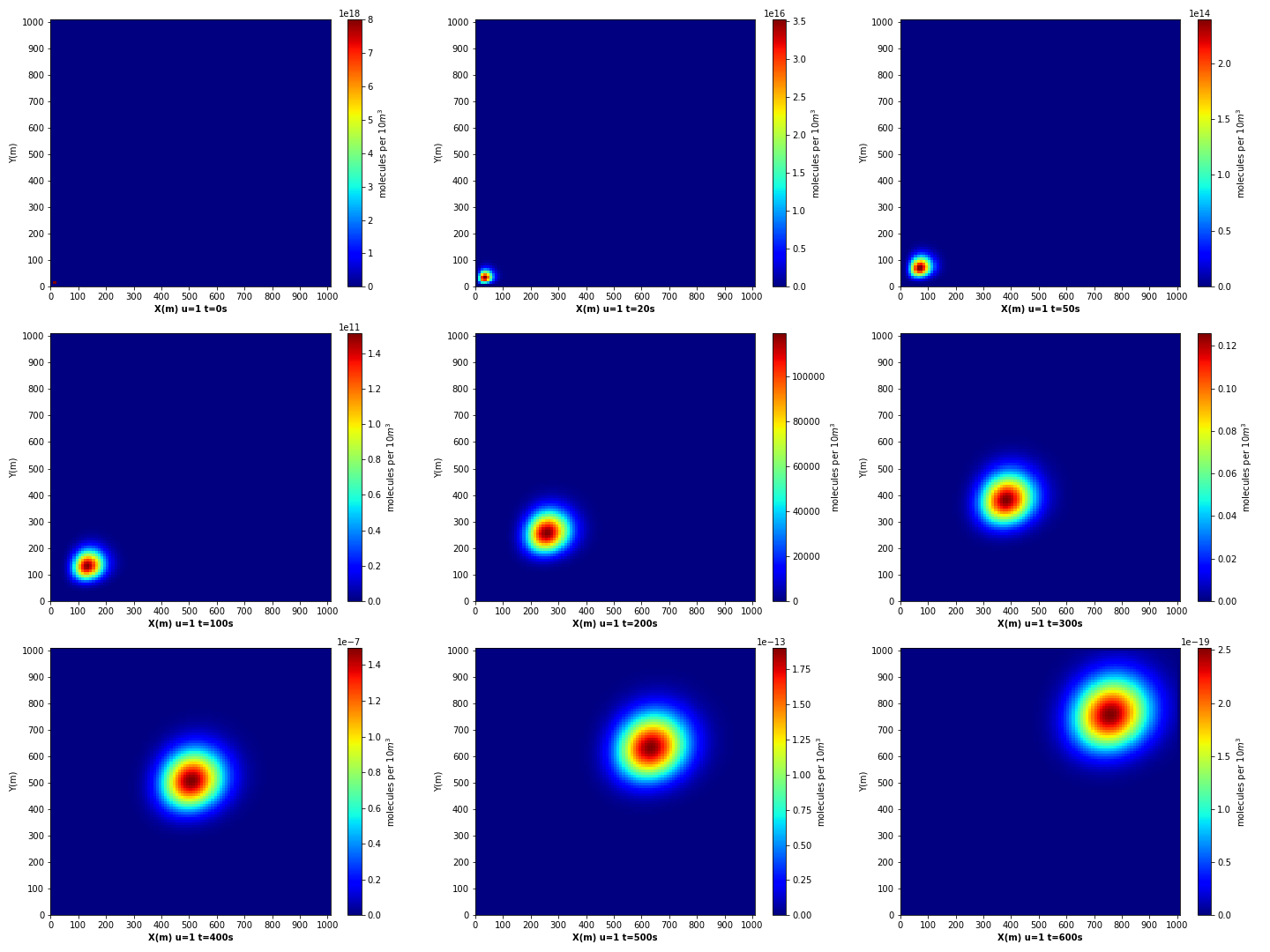}
		\caption{Horizontal slice plots at \( z=1 \) showing \( \mathrm{O}_3 \) concentrations at times \( 0 \, \text{s}, 20 \, \text{s}, 50 \, \text{s}, 100 \, \text{s}, 200 \, \text{s}, 300 \, \text{s}, 400 \, \text{s}, 500 \, \text{s}, \) and \( 600 \, \text{s} \), with \( u = 1 \, \text{m/s} \) and \( k = 0.00002 \, \text{m}^2/\text{s} \).}
		\label{fig:example2}
	\end{figure}
	
	\begin{figure}[H]
		\centering
		\includegraphics[width=1.213\textwidth]{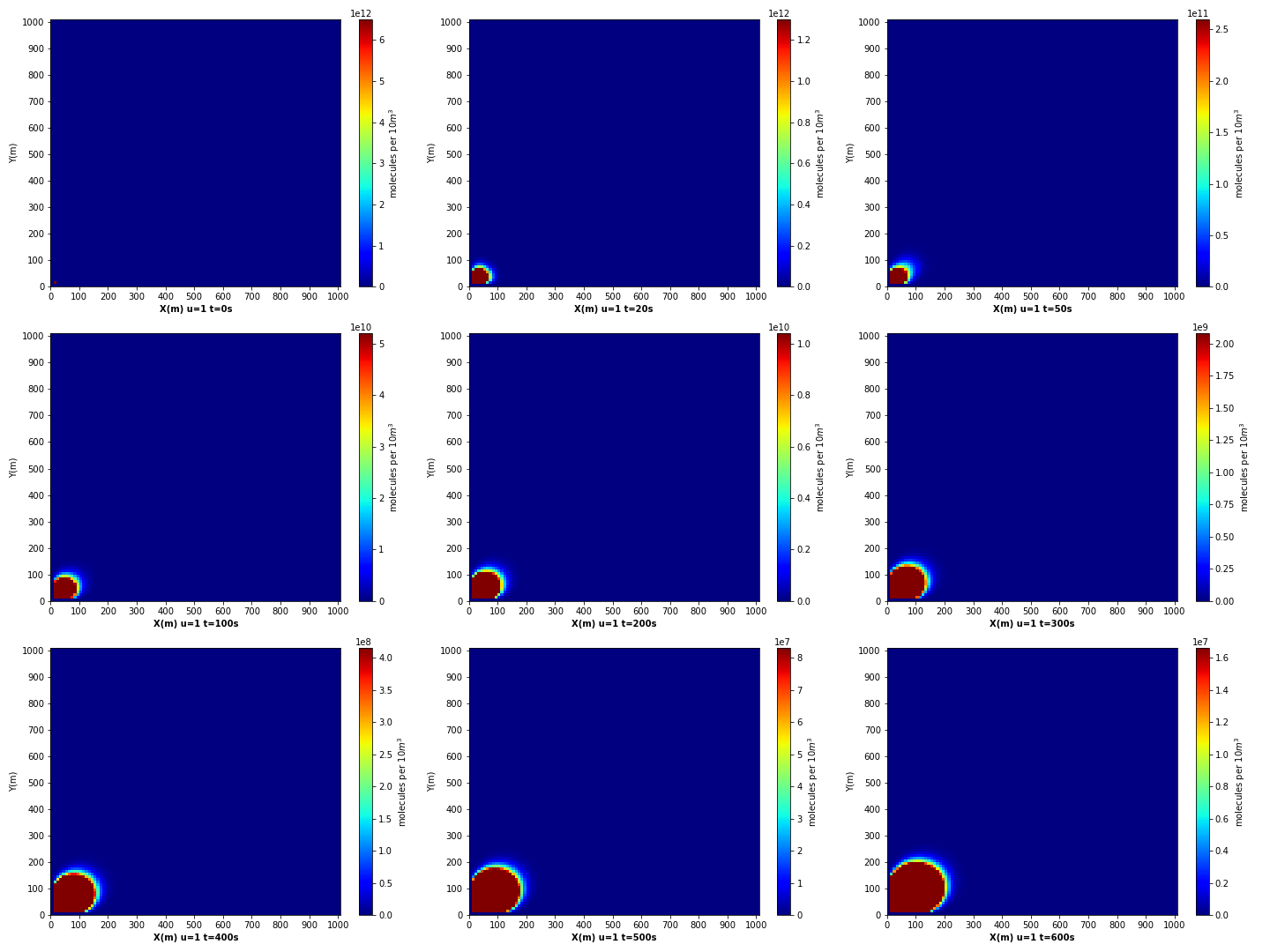}
		\caption{Horizontal slice plots at \( z = 1 \) showing \( \mathrm{NO} \) concentrations at times \( 0 \, \text{s}, 20 \, \text{s}, 50 \, \text{s}, 100 \, \text{s}, 200 \, \text{s}, 300 \, \text{s}, 400 \, \text{s}, 500 \, \text{s}, \) and \( 600 \, \text{s} \), with \( u = 1 \, \text{m/s} \) and \( k = 0.00002 \, \text{m}^2/\text{s} \).}
		\label{fig:example3}
	\end{figure}
\end{landscape}
\subsection{Attractor visualization}
\noindent 

In Figure \ref{Attractor2}, a compelling observation arises as we discern the convergence of all concentrations towards a shared region, accompanied by noticeable oscillations within its confines. This distinct region, identified as our attractor, stands out prominently against the backdrop of surrounding data points. Remarkably, it encapsulates the trajectories of the concentrations of each species, in each grid cell, in function of the other two species through time, serving as a central hub for their dynamic behavior. This emphasizes the pivotal role played by the attractor in dictating the enduring patterns and long-term dynamics of the system, shedding light on its significant influence on the overall system behavior.
\begin{figure}[H]
	\centering
	\includegraphics[scale=0.29]{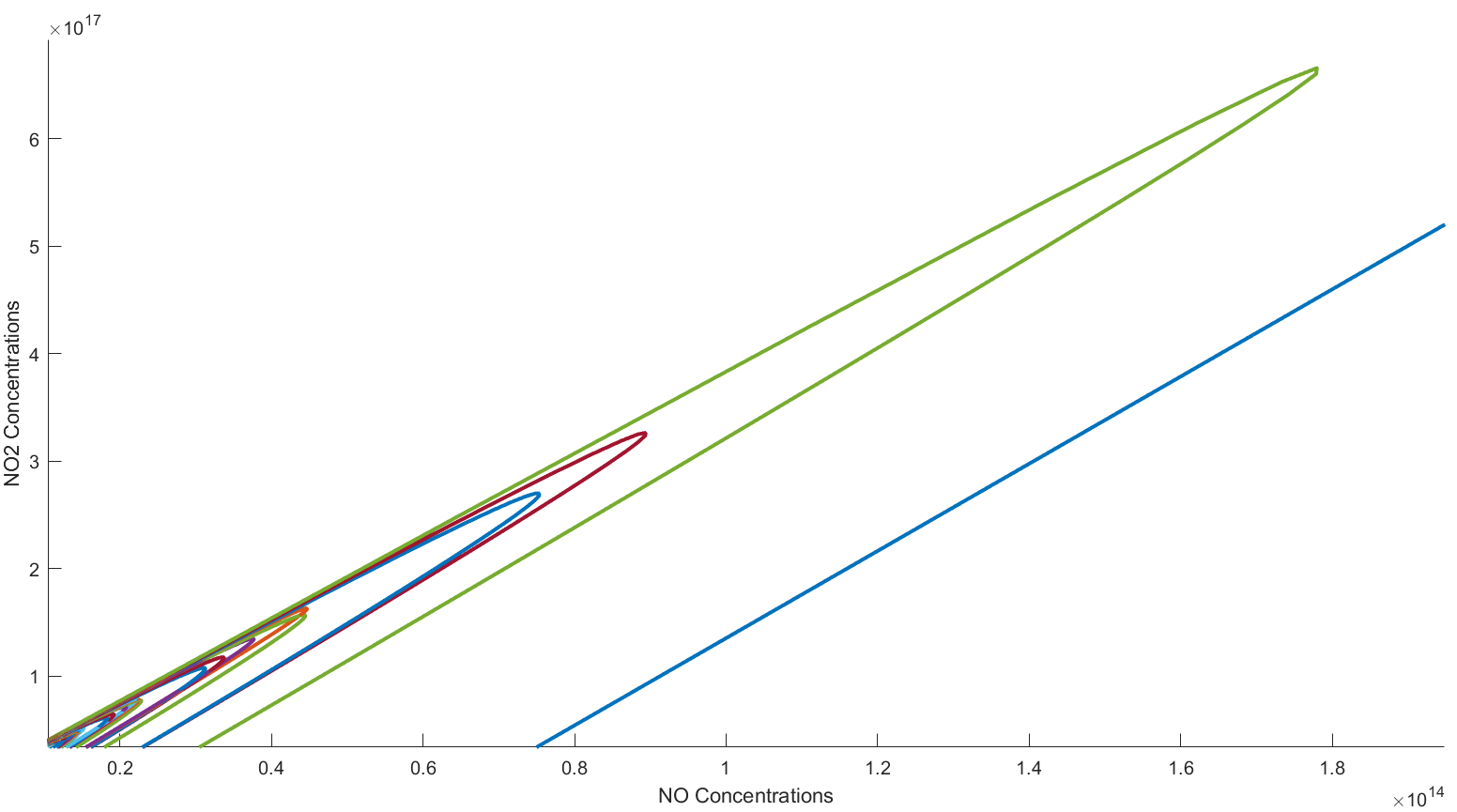}\\
	
	\vspace{2cm}
	\includegraphics[scale=0.29]{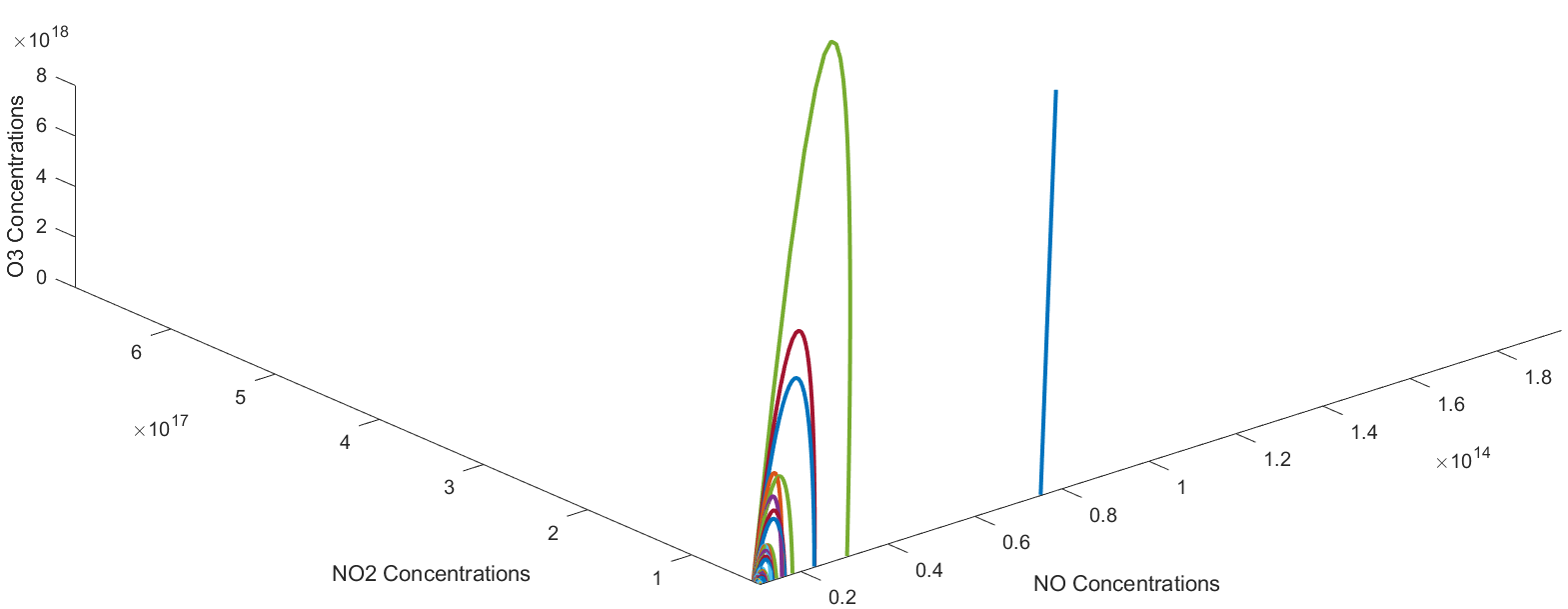}
	\caption{The long-term evolution of the trajectories of the concentration of each species in each cell of the grid.}
	\label{Attractor2}
\end{figure}
\section{Conclusion}
\noindent

In this work, several results concerning the existence, uniqueness, and positivity of solutions as well as the asymptotic behavior through a finite fractal dimensional global attractor for an Advection-Diffusion-Reaction (ADR) equation, have been established. To achieve this goal, various results from semigroup and global attractors theories were used. An analytical solution in 2-D was presented and numerical simulations were conducted using two Explicit Finite Difference schemes: the Explicit Centered scheme to solve the two-dimensional Advection-Diffusion equation and the First Order Upwind numerical scheme to solve the three-dimensional Advection-Diffusion-Reaction equation, which includes reactions involving $\mathrm{NO}$, $\mathrm{NO}_2$, $\mathrm{O}_3$, and $\mathrm{O}_2$, with specified initial and boundary conditions. The resulting plots illustrate the solutions of the two-dimensional equation and later the concentration distribution of the dispersed gases across the horizontal plane at $z=1$.

\section*{CRediT Authors Contribution Statement}
\begin{description}
	\item[M. ELGHANDOURI:] Mathematical proofs, manuscript writing, review, and editing.
	\item[K. EZZINBI:] Mathematical methodology and supervision, manuscript review.
	\item[L. SAIDI:] Analytical solution, numerical simulation, and manuscript writing.
	\item[M. ELGHANDOURI \& L. SAIDI:] Discussion and attractor visualization.
\end{description}

\section*{Acknowledgments}\noindent
The research described in this article was made possible by the support of the I-MAROC project. The authors would like to thank the anonymous referees for their careful reading of the manuscript and their insightful feedback.

\section*{Financial disclosure.}\noindent
The authors assert that they do not possess any identifiable conflicting financial interests or personal relationships that might have seemed to impact the research presented in this paper.

\section*{Conflict of interest.}\noindent
This work does not have any conflicts of interest.


\begin{thebibliography}{00}
\bibitem{Batkai}
A. Batkai, M. K. Fijavz, et A. Rhandi. Positive operator semigroups. Operator Theory: advances and applications, 2017, vol. 257.

\bibitem{Bear}
J. Bear. Hydraulics of Groundwater, Dover, Minneola, 2007.
\bibitem{Bainov}
D. D. Bainov and P. S. Simeonov. Integral inequalities and applications, vol. 57, Springer Science and Business Media, 2013.

\bibitem{Buckman}
N. M. Buckman. The linear convection-diffusion equation in two dimensions.

\bibitem{Burden}
Burden, R., Faires, J., Burden, A.: Numerical analysis, 10th edn. Cengage Learning Inc., Boston (2015).

\bibitem{Czaja}
R. Czaja et M. Efendiev. A note on attractors with finite fractal dimension. Bulletin of the London Mathematical Society, 2008, vol. 40, no 4, \url{https://doi.org/10.1112/blms/bdn044}.

\bibitem{ELAMVAZHUTHI}
K. Elamvazhuthi, H. Kuiper, M. Kawski, et al. Bilinear controllability of a class of advection–diffusion–reaction systems. IEEE Transactions on Automatic Control, 2018, vol. 64, no 6, p. 2282-2297, \url{https://doi.org/10.1109/TAC.2018.2885231}.

\bibitem{Efendiev}
M. Efendiev, Finite and Infinite Dimensional Attractors for Evolution Equations of Mathematical Physics, Gakuto International Series. Mathematical Sciences and Applications, vol. 33, Gakkotosho Co., Ltd., Tokyo, 2010.

\bibitem{Genuchten}
M. T. Van Genuchten, W. J. Alves. Analytical Solutions of the One-Dimensional Convective-Dispersive Solute Transport Equation, U.S. Department of Agriculture, Agricultural Research Service Technical Bulletin, vol. 1661: Government Printing Office; 1982.

\bibitem{Gustafson}
B. Gustafson, H. Kreiss, and J. Oliger. Time Dependent Problems and Difference Methods, John Wiley \& Sons, New York, 1995.

\bibitem{Hale}
J. Hale. Asymptotic behavior of dissipative systems (providence, ri: American mathematical society) go to reference in article (1988).

\bibitem{Hetrick}
D. K. Hetrick. Dynamics of Nuclear Reactors, University of Chicago, Chicago, 1971.

\bibitem{Hundsdorfer}
W. H. Hundsdorfer, and G. J. Verwer. Numerical solution of time-dependent advection-diffusion-reaction equations. Vol. 33. Berlin: Springer, 2003.

\bibitem{KIM}
Kim, A.S. Complete analytic solutions for convection-diffusion-reaction-source equations without using an inverse Laplace transform. Sci Rep 10, 8040 (2020), \url{https://doi.org/10.1038/s41598-020-63982-w}.

\bibitem{LANSER}
D. Lanser et G. J. Verwer. Analysis of operator splitting for advection–diffusion–reaction problems from air pollution modelling. Journal of computational and applied mathematics, 1999, vol. 111, no 1-2, p. 201-216, \url{https://doi.org/10.1016/S0377-0427(99)00143-0}.

\bibitem{Lewis}
P. E. Lewis and J. P. Ward. The Finite Element Method: Principles and Applications (Addison-Wesley, Wokingham, 1991).

\bibitem{MCLEAN}
W. Mclean, K. Mustapha, R. Ali et al. Well-posedness of time-fractional advection-diffusion-reaction equations. Fractional Calculus and Applied Analysis, 2019, vol. 22, no 4, p. 918-944, \url{https://doi.org/10.1515/fca-2019-0050}.

\bibitem{Mphephu}
N. Mphephu. Numerical solution of 1-D convection-diffusion-reaction equation. M.S. thesis, University of Venda, African Institute for Mathematical Sciences, 2013.

\bibitem{Murray}
J. D. Murray. Mathematical Biology I, Springer-Verlag, Berlin, 2002.

\bibitem{Parhizi}
Parhizi M., Kilaz G., Ostanek J. K., Jain A., Analytical solution of the convection-diffusion-reaction-source (CDRS) equation using Green's function technique, International Communications in Heat and Mass Transfer, Volume 131, 2022, 105869, ISSN 0735-1933,
\url{https://doi.org/10.1016/j.icheatmasstransfer.2021.105869}.

\bibitem{Pazy}
A. Pazy. Semigroups of linear operators and applications to partial differential equations. Vol. 44. Springer Science \& Business Media, 2012.

\bibitem{Peng}
L. Peng, Y. Zhou. The existence of mild and classical solutions for time fractional Fokker–Planck equations. Monatsh Math 199, 377–410 (2022), \url{https://doi.org/10.1007/s00605-022-01710-4}.

\bibitem{Polyanin}
A. D. Polyanin, V. E. Nazaikinskii. Handbook of Linear Partial Differential Equations for Engineers and Scientist. CRC Press Taylor \& Francis Group LLC; 2016.

\bibitem{Schmidt}
H. Schmidt, C. Derognat, R. Vautard et al. A comparison of simulated and observed ozone mixing ratios for the summer of 1998 in Western Europe. Atmospheric Environment, 2001, vol. 35, no 36, p. 6277-6297, \url{https://doi.org/10.1016/S1352-2310(01)00451-4}.

\bibitem{Shih}
T. M. Shih. Numerical Heat Transfer, Springer-Verlag, Berlin, 1984.

\bibitem{Shu}
C. W. Shu. TVD time discretizations, SIAM J. Sci. Stat. Comput., 9 (1988), pp. 1073-1084, \url{https://doi.org/10.1137/090907}.

\bibitem{Steinfeld}
J. I. Steinfeld. Atmospheric chemistry and physics: from air pollution to climate change. Environment: Science and Policy for Sustainable Development, 1998, vol. 40, no 7, p. 26-26, \url{https://doi.org/10.1080/00139157.1999.10544295}.

\bibitem{Zavaleta}
A. U. Zavaleta, A. L. de Bortoli, M. Thompson. Analysis and simulation for a system of chemical reaction equations with a vortex formulation, Applied Numerical Mathematics, Volume 47, Issues 3–4, 2003, Pages 559-573, \url{https://doi.org/10.1016/S0168-9274(03)00083-7}.

\bibitem{ZHANG}
Z. Zhang, M. V. Tretyakov, B. Rozovskii et al. Wiener chaos versus stochastic collocation methods for linear advection-diffusion-reaction equations with multiplicative white noise. SIAM Journal on Numerical Analysis, 2015, vol. 53, no 1, p. 153-183, \url{https://doi.org/10.1137/130932156}.

\bibitem{WHEELER}
M. F. Wheeler, et C. N. Dawson. An operator-splitting method for advection-diffusion-reaction problems. 1987.
\end{thebibliography}
\end{document}